\documentclass{amsart}%
\usepackage{amssymb}
\usepackage{amsfonts}
\usepackage{amsmath}
\usepackage{graphicx}%
\providecommand{\U}[1]{\protect\rule{.1in}{.1in}}
\newtheorem{theorem}{Theorem}
\theoremstyle{plain}

\newtheorem{corollary}{Corollary}

\newtheorem{proposition}{Proposition}
\newtheorem{remark}{Remark}

\numberwithin{equation}{section}
\begin{document}
\title[Infinite lower triangular matrices]{Applications of infinite lower triangular matrices and their group structure
in combinatorics and the theory of orthogonal polynomials. }
\author{Pawe\l \ J. Szab\l owski}
\address{Department of Mathematics and Information Sciences, \\
Warsaw University of Technology\\
ul Koszykowa 75, 00-662 Warsaw, Poland }
\email{pawel.szablowski@gmail.com}
\thanks{The author is grateful to an anonymous referee for his valuable remarks, which
led to improvements in the paper.}
\date{April 2025}

\begin{abstract}
Our focus is on the set of lower-triangular, infinite matrices that have
natural operations like addition, multiplication by a number, and matrix
multiplication. With respect to addition this set forms and abelian group
while with respect to matrix multiplication, the invertivle elements of the
set form a group. The set becomes an algebra (non-commutative in fact) with
unity when all three operations are considered together. We indicate important
properties of the algebraic structures obtained in this way. In particular, we
indicate several sub-groups or sub-rings. Among sub-groups, we consider the
group of Riordan matrices and indicate its several sub-groups. We show a
variety of examples (approximately 20) of matrices that are composed of the
sequences of important polynomial or number families as entries of certain
lower-triangular infinite matrices. New, significant relationships between
these families can be discovered by applying well-known matrix operations like
multiplication and inverse calculation to this representation.

The paper intends to compile numerous simple facts about lower-triangular
matrices, specifically the family of Rionian matrices, and briefly review
their properties.

\end{abstract}
\maketitle

\section{Introduction, notation and elementary observations}

Let $\mathbf{A}\allowbreak\overset{df}{=}\allowbreak\left[  a_{n,j}\right]
_{n,j\geq0,}$ with $a_{nj}\allowbreak=\allowbreak0$ for all $j>n\geq0$, be a
lower-triangular infinite matrix with entries belonging to $\mathbb{C}$, in
general. For reasons that will be obvious in the sequel, let us agree that
both index entries will start from $0$.

To avoid convergence issues, let us agree (following \cite{Szabl14}), that
every infinite lower-triangular matrix will be understood as a sequence
$\left\{  \mathbf{A}_{n}\right\}  $ of finite matrices $\mathbf{A}_{n}$
organized in such a way that the matrix $\mathbf{A}_{n}$ is a sub-matrix of
$\mathbf{A}_{n+1}$ and we have in terms of a block structure
\[
\mathbf{A}_{n+1}\allowbreak=\allowbreak%
\begin{bmatrix}
\mathbf{A}_{n} & 0\\
\mathbf{a}_{n}^{T} & \alpha_{n,n}%
\end{bmatrix}
,
\]
where $\mathbf{a}_{n}^{T}$ is certain row vector of dimension $n+1$ and
$\alpha_{n,n}$ certain complex number. Since index $n$ within this paper runs
usually from $0$ and traditionally indices of row and columns within the
matrix run from $1$, we notice that within this paper $\mathbf{A}_{n}$ will
usually denote a $\left(  n+1\right)  \times\left(  n+1\right)  $ matrix.

We will denote such sequences of matrices by $\mathbf{A\allowbreak
=\allowbreak}\left\{  \mathbf{A}_{n}\right\}  $, meaning that we deal with a
lower-triangular infinite matrix $\mathbf{A}$ whose $\left(  n+1\right)
\times\left(  n+1\right)  $ matrices in their upper left corners are matrices
$\mathbf{A}_{n}$, $n\allowbreak=\allowbreak1,\ldots$. This understanding of
infinite lower-triangular matrices helps address some of the reviewer's
concerns about the existence and convergence of matrix operations, since all
operations on lower-triangular matrices of infinite dimensions can be
understood as operations on sequences of finite matrices or vectors. No
convergence considerations are necessary. Notice also that in accordance with
our index convention, we have $\mathbf{A}_{1}\allowbreak=\allowbreak
\alpha_{0,0}\allowbreak\overset{df}{=}\allowbreak\alpha_{0}$. Let us extend
this convention and denote all diagonal elements $a_{n,n}$ by a single index
that is $a_{n}$ if of course $\allowbreak\mathbf{A\allowbreak=\allowbreak
}\left[  a_{k,j}\right]  $. The notation $\mathbf{A\allowbreak=\allowbreak
}\left[  a_{n,j}\right]  _{n,j\geq0,}$ and $\mathbf{A\allowbreak=\allowbreak
}\left\{  \mathbf{A}_{n}\right\}  ,$ where $\mathbf{A}_{n}\allowbreak
=\allowbreak\lbrack a_{nj}]_{n,j=0,\ldots,n}$ shall be used alternatively,

In the sequel, we will use the simplified convention: $\mathbf{A}%
_{n}\allowbreak=\allowbreak\lbrack a_{k,j}]_{n}$. It is elementary to notice
that such matrices equipped with scalar multiplication and matrix addition,
constitute a linear space. Similarly, the set of infinite, lower-triangular
matrices is a non-commutative, associative algebra that uses both matrix
addition and multiplication, as well as scalar multiplication. More precisely,
we have $\alpha\mathbf{A\allowbreak+\allowbreak}\beta\mathbf{B\allowbreak
=\allowbreak}\left[  \alpha a_{n,j}+\beta b_{n,j}\right]  $ if
$\mathbf{A\allowbreak=\allowbreak\lbrack}a_{n,j}]$ and $\mathbf{B}%
\allowbreak=\allowbreak$ $\left[  b_{n,j}\right]  $, $\alpha$, $\beta
\in\mathbb{C}$ and $\mathbf{AB}\allowbreak=\allowbreak\left[  \sum_{k=j}%
^{n}a_{n,k}b_{k,j}\right]  $. Let us notice that the set of diagonal matrices
makes a commutative sub-algebra of our algebra. Let us agree that diagonal
matrices will be denoted in the following way. Namely, $\left[  \left\{
\beta_{n}\right\}  \right]  $ will denote diagonal matrix with $\beta_{n}$ as
its ($n+1)\times(n+1)$ entry. Notice that with respect to matrix addition and
matrix multiplication the set of lower-triangular matrices constitutes a
non-commutative ring. Among these diagonal matrices the ones with all entries
equal to zero serve as the zero element of the ring and the diagonal matrix
$\left[  \left\{  1\right\}  \right]  $ (i.e., with all diagonal elements
equal to $1)$ as the ring's "one", the unity. These special matrices will be
denoted respectively $\mathbf{0}$ and $\mathbf{1}$. Let us denote by
$\mathcal{S}$ the whole algebra of lower-triangular matrices and by
$\mathcal{D}$ the sub-algebra of all diagonal matrices.

Another important sub-algebra of $\mathcal{S}$ comprises of all
lower-triangular with zeros on their diagonal. Moreover, this sub-algebra
considered as a sub-ring is an ideal.

Finally, let us consider lower-triangular matrices $\left[  a_{n,j}\right]  $
with $a_{n,j}\allowbreak=\allowbreak0$ whenever $n-j$ is an odd number. Let us
denote the set of such matrices by $\mathcal{SE}$. To complete introducing the
notation let us agree that matrices of the size $n\times1$, i.e., columns will
also be called vectors and the infinite column (raw) vector $\mathbf{a}$ will
be also understood as sequence $\left\{  \mathbf{a}_{n}\right\}  _{n\geq0}$ of
$(n+1)\times1$ vectors $\mathbf{a}_{n}.$ All .finite and infinite vectors will
be denoted generally by the bold lower case letters, , i.e., $\mathbf{a,}$
$\mathbf{b}$, and so on. The symbol $^{T}$ denotes transposition of a matrix.
For example $\mathbf{a}^{T}$ denotes a row matrix of the size $1\times n.$

We have an elementary observation

\begin{proposition}
$\mathcal{SE}$ is a sub-algebra of $\mathcal{S}$.
\end{proposition}

\begin{proof}
Multiplication by a number retains the property of an entry of being $0$
whenever $n-j$ is odd. Similarly, with the sum of such matrices. Now let us
recall that $\sum_{k=j}^{n}a_{n,k}b_{k,j}$ is a $(n,j)-$th entry of a product
of two matrices $\left[  a_{n,j}\right]  $ and $\left[  b_{n,j}\right]  $.
Observe that $n-j\allowbreak=\allowbreak n-k\allowbreak+\allowbreak k-j$ for
all $j\allowbreak\leq k\allowbreak\leq\allowbreak n$. Now, if $n-j$ is odd
then either $n-k$ or $k-j$ must also be odd. Hence, if $\left[  a_{n,j}%
\right]  $ and $\left[  b_{n,j}\right]  $ both belong to $\mathcal{SE}$ their
product must also belong to $\mathcal{SE}$.
\end{proof}

Examples of this sub-algebra are presented in Subsubsection \ref{Be&Eu}.

As mentioned above, with respect to the addition of matrices, $\mathcal{S}$ is
not only a commutative (i.e., Abelian) group but also a linear space, if one
considers also multiplication by a number. However, with respect to the matrix
multiplication those lower-triangular matrices form a non-commutative monoid (
i.e., magma with associative operation and identity). Notice that the identity
of this monoid it is matrix $\mathbf{1}$ , i.e., the diagonal matrix with $1$
on its diagonal.

Using this observation and the well-known formulae for the block
multiplication of matrices we see that
\begin{equation}
\mathbf{A}_{n+1}\mathbf{B}_{n+1}=%
\begin{bmatrix}
\mathbf{A}_{n} & 0\\
\mathbf{a}_{n}^{T} & \alpha_{n}%
\end{bmatrix}%
\begin{bmatrix}
\mathbf{B}_{n} & 0\\
\mathbf{b}_{n}^{T} & \beta_{n}%
\end{bmatrix}
=%
\begin{bmatrix}
\mathbf{A}_{n}\mathbf{B}_{n} & 0\\
\mathbf{a}_{n}^{T}\mathbf{B}_{n}+\alpha_{n}\mathbf{b}_{n}^{T} & \alpha
_{n}\beta_{n}%
\end{bmatrix}
. \label{mn}%
\end{equation}

It is also elementary to notice that ring $\mathcal{S}$ has divisors of zero.
Hence, it is not a domain. The following two matrices are the examples of the
left and right divisors of zero in this ring with non-zero elements on the
diagonal:%
\[%
\begin{bmatrix}
a_{0} & 0 & 0\\
a_{1} & 0 & 0\\
a_{2} & 0 & 0
\end{bmatrix}
,~~%
\begin{bmatrix}
0 & 0 & 0\\
0 & b_{1} & 0\\
b_{2} & b_{3} & b_{4}%
\end{bmatrix}
.
\]
On the other hand, matrices having a non-zero element on the $(0,0)$ position,
i.e., in the upper left corner cannot be right divisors of zero.

Notice also that diagonal matrices with the same number, say $\alpha$, on the
diagonal, i.e., $\left[  \left\{  \alpha\right\}  \right]  $ can be identified
with this number since from the formula (\ref{mn}) it follows that
\[
\lbrack a_{n,j}][\left\{  \alpha\right\}  ]=[\alpha a_{n,j}].
\]

Now, let us discuss the set of matrices $\mathbf{A}$ which are invertible,
i.e., such matrices for which there exists a matrix (called inverse) and
denoted by $\mathbf{A}^{-1}$ such that
\[
\mathbf{AA}^{-1}\allowbreak=\allowbreak\mathbf{A}^{-1}\mathbf{A\allowbreak
=\allowbreak1.}%
\]
Notice, that if a matrix from $\mathcal{S}$, say $\mathbf{A}$ has
representation $\left\{  \mathbf{A}_{n}\right\}  $ then the diagonal elements
of each of the matrices $\mathbf{A}_{n}$ are their eigenvalues. Consequently,
each matrix $\mathbf{A}_{n}$ having non-zero elements on the diagonal is
invertible. So the set of invertible elements of $\mathbf{\mathcal{S}}$
defined by the conditions: $\mathbf{A\allowbreak=\allowbreak\lbrack}a_{n,j}]$
is invertible iff $\forall j\geq0:$
\[
\text{ }a_{j}\neq0.
\]
Let us denote by $\mathcal{I}$ the set of invertible matrices. It is
elementary to notice that set of all invertible matrices forms a linear cone,
i.e., if $\mathbf{A}$ is in $\mathcal{I}$ then all matrices of a form
$\mathbf{A[}\left\{  \beta\right\}  ]\in\mathcal{I}$ for all $\beta\neq0$. The
family of such matrices forms the so-called skew field (or a division ring)
with the the skew field operations being the ordinary matrix addition
(commutative) and (usually non-commutative) multiplication.

The other properties of elements of $\mathcal{S}$ considered as linear
operators, i.e., their eigenvectors and some decompositions at least for the
special, more precisely Riordan matrices, are presented in a recently
published paper \cite{Cheo22}.

Thus, naturally all elements of $\mathcal{S}$ that have non-zero elements on
their diagonal, form a group if we confine ourselves to multiplication. We
denote this group by $\mathcal{L}$.

As it follows from, Wikipedia (inverted blockwise formula), also e.g. from
\cite{Bern09} or from direct calculation following (\ref{mn}), the formula for
inversion of block matrices applied to the special case when lower-right-most
corner matrix has dimension $1\times1$ yields%
\begin{equation}
\mathbf{A}_{n+1}^{-1}=%
\begin{bmatrix}
\mathbf{A}_{n}^{-1} & 0\\
-\mathbf{a}_{n}^{T}\mathbf{A}_{n}^{-1}/\alpha_{n} & \alpha_{n}^{-1}%
\end{bmatrix}
. \label{odwrt}%
\end{equation}

As a corollary we have the following observation.

\begin{remark}
\label{proste}1. If matrix $\mathbf{A}$ has all integer entries and moreover
it has $1$ on its diagonal, then matrix $\mathbf{A}^{-1}$ has also integer
entries and $1^{\prime}s$ on its diagonal.

2. If matrix $\mathbf{A}$ has polynomial entries except for the diagonal whose
entries are numbers, then its inverse also has polynomial entries.

3. If $\mathbf{A\allowbreak=\allowbreak}\left[  a_{nj}\right]  $ and
$\mathbf{A}^{-1}\allowbreak=\allowbreak\left[  b_{nj}\right]  $, then also the
following two matrices are inverses of one another:

$\qquad$3a. $\forall\lambda\in\mathbb{C}:\mathbf{F(}\lambda)\allowbreak
\mathbf{=\allowbreak}\left[  a_{nj}\lambda^{n-j}\right]  $ and $\mathbf{B(}%
\lambda\mathbf{)}\allowbreak=\allowbreak\left[  b_{nj}\lambda^{n-j}\right]  .$

\qquad3b. Suppose $\left\{  \alpha_{n}\right\}  _{n\geq0}$, $\alpha_{n}%
\neq0:\mathbf{\tilde{A}\allowbreak=\allowbreak}\left[  \alpha_{n}%
a_{n,j}\right]  $ and $\mathbf{\tilde{B}\allowbreak=\allowbreak}\left[
b_{n,j}/\alpha_{j}\right]  .$

4.
\[
\left[  \left\{  \alpha_{n}\right\}  \right]  \left[  a_{n,j}\right]  \left[
\left\{  \beta_{j}\right\}  \right]  =\left[  \alpha_{n}a_{n,j}\beta
_{j}\right]  ,
\]
In particular%
\begin{equation}
\left[  \left\{  \alpha_{n}\right\}  \right]  \left[  a_{n,j}\right]  \left[
\left\{  \alpha_{j}\right\}  \right]  ^{-1}=\left[  \alpha_{n}\tilde{a}%
_{n,j}/\alpha_{j}\right]  , \label{mnoz}%
\end{equation}
where $\tilde{a}_{n,j}$ is defined by the relationship $\left[  a_{n,j}%
\right]  ^{-1}\allowbreak=\allowbreak\left[  \tilde{a}_{n,j}\right]  $.

5. Suppose $\mathbf{A\allowbreak=\allowbreak}\left[  d_{n-j}\right]  $, for
some sequence $\left\{  d_{j}\right\}  $, with $d_{0}\allowbreak=\allowbreak1
$, then $\mathbf{A}^{-1}\allowbreak=\allowbreak\left[  e_{n-j}\right]  $,
where sequence $\left\{  e_{j}\right\}  $ is defined by the recursion:%
\[
e_{0}=1,~~e_{k}=-\sum_{s=0}^{k-1}d_{k-s}e_{s}.
\]

\end{remark}

Further the paper contains Section \ref{PS} devoted to the definition of the
formal power series and presentation of some of its properties, Section
\ref{Rio} devoted to the presentation of the group of Riordan matrices and
some of its subgroups. Finally we have Sections \ref{Ex} containing many
subsections in which examples are grouped with respect to their source. The
last Section \ref{GL} contains glossary of various algebraic terms and symbols
use in the paper.

\section{Formal power series\label{PS}}

In the sequel, we will consider also some sequences and their generating
functions. In general, these sequences as well as variables in the formal
power series being the expansion of these GF can be complex. However, in
almost all cases, we will work with real sequences as well as real GF's. That
is why we will assume that all variables and numbers will be real. We will
also be aware that without any difficulty the results can be extended to the
complex case.

In order to move further, let us extend the usual definition of the so-called
generating function (GF) of a sequence $\left\{  a_{n}\right\}  _{n\geq0}$.
Since we will mostly deal with infinite matrices predominantly
lower-triangular, let us express the notion of GF in terms of matrix
operation. This perspective is not very revolutionary, however, it simplifies,
in the author's opinion, many concepts and ideas. Hence, the ordinary sequence
$\left\{  a_{n}\right\}  _{n\geq0}$ will be viewed either as a column vector
with $a_{n}$ as its $n+1-$st entry or as the diagonal matrix $\left[  \left\{
a_{n}\right\}  \right]  $. The modification of the notion of the GF is as
follows. First, let us fix the so-called "reference sequence" or the so-called
"denominator sequence " $\left\{  c_{n}\right\}  _{n\geq0}$, i.e., some
sequence such that $c_{0}\allowbreak=\allowbreak1$ and $c_{n}\neq0$, for
$n\geq1$. Then, by the generating function (GF) of a sequence $\left\{
a_{n}\right\}  $, given the reference sequence $\left\{  c_{n}\right\}  $, (or
briefly $\left\{  c_{n}\right\}  $ -GF) the formal power series
\[
F_{a}^{c}(x)=\sum_{n\geq0}a_{n}x^{n}/c_{n}.
\]

\begin{remark}
Let us fix the reference sequence $\left\{  c_{n}\right\}  $. Let us consider
two sequences $\left\{  a_{n}\right\}  $ and $\left\{  b_{n}\right\}  $ with
$\left\{  c_{n}\right\}  $ -GF's respectively $F_{a}^{c}(x)$ and $F_{b}%
^{c}(x)$, then $F_{a}^{c}(x)F_{b}^{c}(x)$ is the $\left\{  c_{n}\right\}  $
-GF of the following sequence $\left\{  \sum_{j=0}^{n}%
\genfrac{\langle}{\rangle}{0pt}{}{n}{j}%
_{c}a_{j}b_{n-j}\right\}  $, where we have denoted $%
\genfrac{\langle}{\rangle}{0pt}{}{n}{j}%
_{c}\allowbreak=\allowbreak\frac{c_{n}}{c_{j}c_{n-j}}.$
\end{remark}

We will call $%
\genfrac{\langle}{\rangle}{0pt}{}{n}{j}%
_{c}$ the $\left\{  c_{n}\right\}  $ -binomial coefficient. Let us set also $%
\genfrac{\langle}{\rangle}{0pt}{}{n}{j}%
_{c}\allowbreak=\allowbreak0$ for $j>n$ and $j<0$. We will also call the
following power series $C(x)\allowbreak=\allowbreak\sum_{n\geq0}x^{n}/c_{n}$,
the reference GF, that is briefly RGF.

\begin{remark}
If $c_{n}\allowbreak=\allowbreak1$ , $n\geq0$ and $\left\vert x\right\vert
<1$, the RGF $C(x)\allowbreak=\allowbreak1/(1-x)$, when $c_{n}\allowbreak
=\allowbreak n!$ we get $C(x)\allowbreak=\allowbreak\exp(x)$ while when
$c_{n}\allowbreak=\allowbreak1/\prod_{j=1}^{n}(1-q^{j})$, for some
$q\in(-1,1),$ $n\geq1$ and $\left\vert x\right\vert <1$, then $C(x)\allowbreak
=\allowbreak1/\prod_{j=0}^{\infty}(1-xq^{j})$ as it follows from the so-called
$q-$binomial theorem.
\end{remark}

Notice that this formal power series can be obtained as the result of the
ordinary matrix operations. Namely, we have
\[
F_{a}^{c}(x)=\sum_{n\geq0}a_{n}x^{n}/c_{n}=\mathbf{1}^{T}[\left\{
a_{n}\right\}  ]\left[  \left\{  1/c_{n}\right\}  \right]  \mathbf{x,}%
\]
where we denoted by $\mathbf{1}$ and $\mathbf{x}$ the column vectors with
$n-$th entries respectively $1$ and $x^{n}$. When, $c_{n}\allowbreak
=\allowbreak1$ then we talk about ordinary GF or simply GF of the sequence
$\left\{  a_{n}\right\}  $, when $c_{n}\allowbreak=\allowbreak n!$ then we
talk about "exponential" GF of the sequence $\left\{  a_{n}\right\}  $,
finally when $c_{n}\allowbreak=\allowbreak\prod_{j=1}^{n}(1-q^{j})$, or
$c_{n}\allowbreak=\allowbreak(\prod_{j=1}^{n}(1-q^{j}))/(1-q)^{n}$ for
$n>0,$and some $q\in(-1,1)$, then we talk about $q$ GF of the sequence
$\left\{  a_{n}\right\}  $. Notice, that $\lim_{q\rightarrow1^{-}}(\prod
_{j=1}^{n}(1-q^{j}))/(1-q)^{n}\allowbreak=\allowbreak n!$ . This and many
other facts concerning $q-$series can be found in first chapter of \cite{KLS}.

Let us notice that for every lower-triangular matrix, say $\mathbf{A}%
\allowbreak=\allowbreak\lbrack a_{n,i}]\mathbf{\in}S$ and let us fix the
reference sequence $\left\{  c_{n}\right\}  _{n\geq0}$. We can define now a
formal power series
\[
\mathcal{A}(x,y)=\sum_{i=0}^{\infty}\sum_{n=i}^{\infty}x^{i}y^{n}%
a_{n,i}/\left(  c_{i}c_{n}\right)  =\sum_{i=0}^{\infty}x^{i}g_{i}(y)/c_{i},
\]
where $g_{i}(y)\allowbreak=\allowbreak\sum_{n\geq i}y^{n}a_{n,i}/c_{n}$ is the
GF of the $i-$th column of the matrix $\mathbf{A}$. $\mathcal{A}(x,y)$ will be
called a GF of the matrix $\mathbf{A}$.

The notation used in formal power series theory (FPS) will be utilized in the
sequel. Namely, if $p(x)\allowbreak=\allowbreak\sum_{n\geq0}p_{n}x^{n} $ is a
FPS, then $p_{n}\allowbreak$ is often denoted by: $\left[  \left[
x^{n}\right]  \right]  p(x).$

In particular, we have $\sum_{n\geq0}x^{n}[\left[  x^{n}\right]
](p(x)\allowbreak=\allowbreak p(x).$ We are using a double square bracket here
to denote "coefficient operator" in order to minimize confusion. Recall that
we use a single square bracket to denote infinite matrices.

\section{Riordan arrays\label{Rio}}

One of the most important subgroups of $\mathcal{S}$ is the so-called Riordan
group $\mathcal{R}$ defined in many positions of literature through the
properties of generating functions (GF) of the entries of elements of
$\mathcal{S}$. Let us fix the reference sequence $\left\{  c_{n}\right\}
_{n\geq0}$

The Riordan group is characterized by the fact that the GF of the $i$-th
column has a form $\forall i\geq0:g_{i}(y)\allowbreak=\allowbreak
f(y)h^{i}(y)/c_{i}$, for some formal power series $f(y)\allowbreak=\allowbreak
f_{0}+\sum_{j\geq1}f_{j}y^{j}/c_{j}$ and $h(y)\allowbreak=\allowbreak
\sum_{j\geq1}h_{j}y^{j}/c_{j}$, with $f_{0},h_{1}\neq0$. For a
lower-triangular infinite matrix being an element of the Riordan group, we
thus have%

\[
\mathcal{A}(x,y)=f(y)\sum_{i=0}^{\infty}x^{i}h^{i}(y)/c_{j}=f(y)C(xh(y)),
\]
where $C(x)$ is the RGF. Hence, the Riordan matrix is characterized by the two
generating functions $f$ and $h$. It is traditionally denoted as $\left(
f,h\right)  $. With the Riordan matrix $(f,h)$ we associate the following
lower-triangular matrix $\left[  d_{n,j}\right]  $, where
\[
d_{n,j}\allowbreak=\allowbreak c_{n}[\left[  x^{n}\right]  ]f(x)h^{j}%
(x)/c_{j}.
\]

Indeed, we get the following function as GF of the $j-$th column%
\[
\sum_{n\geq j}d_{n,j}x^{n}/c_{n}\allowbreak=\allowbreak\sum_{n\geq0}x^{n}[
\left[  x^{n}\right]  ]f(x)h^{j}(x)/c_{j}=f(x)h^{j}(x)/c_{j}.
\]

It is well-known that the product of two Riordan matrices is a Riordan matrix.
Hence, the Riordan matrices form a subgroup $\mathcal{R}$ of group
$\mathcal{S}$. Unfortunately, $\mathcal{R}$ is not a ring since in general a
sum of two Riordan matrices is not Riordan. Many papers have been written
about Riordan matrices over the years. They presented many features of this
group. To ensure completeness in the paper, we will revisit some of these
results, particularly those related to the group structure of the Riordan group.

We start with the following well-known, old result.

\begin{theorem}
[Roman]\label{Roman} Suppose we have two Riordan matrices: $(a(x),b(x))$ and
$(c(x),d(x))$ where functions $a,b,c,d$ are such that $a(0)\allowbreak
\neq\allowbreak0\neq\allowbreak c(0)$, $b(0)\allowbreak=\allowbreak
d(0)\allowbreak=\allowbreak0$ and $b^{\prime}(0)\neq\allowbreak0\allowbreak
\neq\allowbreak d^{\prime}(0)$. Then their product is a Riordan matrix
$(a(x)c(b(x)),d(b(x))$ and its inverse the Riordan matrix $(a(x),b(x))$ is
equal to $\left(  1/a(\bar{b}(x),\bar{b}(x)\right)  $, where the function
$\bar{b}$ is defined by the relationship: $\bar{b}(b(x))\allowbreak
=\allowbreak b\left(  \bar{b}\left(  x\right)  \right)  \allowbreak
=\allowbreak x$.
\end{theorem}

\begin{proof}
The proof for $c_{n}\allowbreak=\allowbreak1$ is presented in \cite{Rom84},
p.43. The proof for the general reference sequence $\left\{  c_{n}\right\}  $
is presented in \cite{GoodLav15}.
\end{proof}

Below, we present a list of known and newly identified subgroups of the
$\mathcal{R}$ group. The most common and the most important reference sequence
is $\left\{  1\right\}  $. That is why it will be considered in the sequel.
The other cases of the reference sequence will be clearly underlined. One of
the nontrivial examples of usage of the "exponential" reference sequence is
given in Proposition \ref{exp} in Section \ref{Ex}. The facts presented in the
list below are predominantly known and scattered through the literature. We
recall them for completeness of the paper. It should be noted that the
subgroup $\mathcal{IP}(p,\beta,\alpha)$ and several of its subgroups listed in
the initial positions below are believed to be unknown and are being
introduced for the first time.

\bigskip

{\huge List of subgroups ...of some subgroups of the Riordan group} \emph{
\label{subgr}}

\begin{enumerate}
\item Riordan matrices of the form $\left(  p(x),\beta x/\left(  1-\alpha
x\right)  \right)  $, for some reals $\alpha$, $\beta,$ $\beta\neq0$ and a
formal power series $p\left(  x\right)  $ such that $p(0)\neq0$ forms a
subgroup of $\mathcal{R}$. The multiplication rule within this subgroup is the
following:
\begin{align*}
&  \left(  p_{1}(x),\frac{\beta_{1}x}{\left(  1-\alpha_{1}x\right)  }\right)
\left(  p_{2}(x),\frac{\beta_{2}x}{(1-\alpha_{2}x)}\right) \\
&  =\left(  p_{1}(x)p_{2}\left(  \frac{\beta_{1}x}{(1-\alpha_{1}x)}\right)
,\frac{\beta_{1}\beta_{2}x}{\left(  1-(\alpha_{1}+\beta_{1}\alpha
_{2})x\right)  }\right)  .
\end{align*}
We will denote this subgroup by $\mathcal{IP}(p,\beta,\alpha)$.

\item Notice, that Riordan matrices from $\mathcal{IP}(p,1,\alpha)$, i.e.,
Riordan matrices of the form $\left(  p(x),x/(1-\alpha x)\right)  $, for some
real $\alpha$, and a formal power series $p\left(  x\right)  $ such that
$p(0)\neq0$, form another subgroup of $\mathcal{R}$. The multiplication rule
within this subgroup is the following:
\begin{align*}
&  \left(  p_{1}(x),\frac{x}{(1-\alpha_{1}x)}\right)  \left(  p_{2}%
(x),\frac{x}{(1-\alpha_{2}x)}\right) \\
&  =\left(  p_{1}(x)p_{2}\left(  \frac{x}{(1-\alpha_{1}x)}\right)  \right)
,\frac{x}{(1-(\alpha_{1}+\alpha_{2})x)}).
\end{align*}
We will denote this subgroup by $\mathcal{P}(p,\alpha)$ and call matrices of
this form a \textbf{generalized Pascal} matrices (briefly GP). Recall that the
name of a \textbf{Pascal} matrix is attributed traditionally to the matrix
$(1/(1-x),x/(1-x))$, hence a particular element of the group $\mathcal{P}%
(p,\alpha)$.

\item Notice that also matrices from $\mathcal{IP}(p,a,0)$, i.e., Riordan
matrices of the form $\left(  p(x),ax\right)  $, for some real $a$ and a
formal power series $p(x)$ such that $p(0)\neq0$, form another subgroup of
$\mathcal{R}$. The multiplication rule within this subgroup is the following:
\[
(p_{1}(x),a_{1}x)(p_{2}(x),a_{2}x)=\left(  p_{1}(x)p_{2}(a_{1}x\right)
,a_{1}a_{2}x).
\]
We will denote this subgroup by $\mathcal{L(}p,a)$. Obviously, $\mathcal{IP}%
(p,a,0)\allowbreak=\allowbreak\mathcal{L(}p,a)$. It has a subgroup
$\mathcal{L(}p,1)$ which we will denote by $\mathcal{A}$ (called
\textbf{Appell} subgroup) or $\mathcal{A(}p)$, if one wants to underline the
dependence on the formal power series $p$. Moreover, $\mathcal{A}$ has a
subgroup consisting of elements such that $p(0)\allowbreak=\allowbreak1$. Let
us denote it by $\mathcal{O}$. Notice that, if $p(x)\allowbreak=\allowbreak1$
then this particular element, i.e., matrix $(1,x)$ plays the r\^{o}le of $1$
in all these groups and subgroups.

\item Riordan matrices of the form $\left(  1,p(x)\right)  $ make another
subgroup of $\mathcal{R}$ ( called \textbf{associated}, sometimes called also
iterations (see \cite{GoodLav15}) subgroup or simply \textbf{Lagrange}
(denoted by $\mathcal{C}$)) with the following multiplication rule:
\[
(1,p_{1}(x))(1,p_{2}(x))=(1,p_{2}(p_{1}(x))).
\]

\item The \textbf{Bell} subgroup can be formed by Riordan matrices of the form
$(g(x),xg(x))$ or $(f(x)/x,f(x))$, with the following multiplication rule:%
\[
\left(  g_{1}(x),xg_{1}(x)\right)  (g_{2}(x),xg_{2}(x))=(g_{1}(x)g_{2}%
(xg_{1}(x)),xg_{1}(x)g_{2}(xg_{1}(x))).
\]

\item The Riordan matrices of the form $\left(  f^{\prime}\left(  x\right)
,f(x)\right)  $ make also the subgroup (called the \textbf{derivative}
subgroup with the following multiplication rule:%
\[
\left(  f^{\prime}\left(  x\right)  ,f\left(  x\right)  \right)  \left(
g^{\prime}\left(  x\right)  ,g\left(  x\right)  \right)  =\left(  f^{\prime
}\left(  x\right)  g\left(  f^{\prime}\left(  x\right)  \right)  ,g\left(
f\left(  x\right)  \right)  \right)  .
\]

\item The Riordan matrices of the form $\left(  f,g\right)  $, where $f$ is
even while $g$ odd function form another sub-group (called
\textbf{Checkerboard}) with the following multiplication rule:%
\[
\left(  f_{1},g_{1}\right)  \left(  f_{2},g_{2}\right)  =\left(  f_{1}%
f_{2}(g_{1}),g_{2}(g_{1}\left(  x)\right)  \right)  .
\]
This is so, since $f_{1}\left(  -x)\right)  f_{2}(g_{1}\left(  -x\right)
)\allowbreak=\allowbreak f_{1}\left(  x\right)  f_{2}\left(  -g_{1}\left(
x\right)  \right)  \allowbreak=\allowbreak f_{1}\left(  x\right)  f_{2}\left(
g_{1}\left(  x\right)  \right)  $ is even if only $f_{1}$, $f_{2}$ are even
and $g_{1}$ is odd. Similarly, we have $g_{2}\left(  g_{1}\left(  -x\right)
\right)  \allowbreak=\allowbreak g_{2}\left(  -g_{1}\left(  x\right)  \right)
\allowbreak=\allowbreak-g_{2}\left(  g_{1}\left(  x\right)  \right)  $. Hence,
$g_{2}\left(  g_{1}\right)  $ is odd if only $g_{2}$ and $g_{1}$ are odd.
\end{enumerate}

We have the following couple of simple observations concerning subgroups of
$\mathcal{R}$.

\begin{remark}
1. We have:
\[
\mathcal{S\supset R\supset\mathcal{IP}}(p,a,\alpha)\mathcal{\supset
\mathcal{P}}(p,\alpha)\mathcal{\supset A}(p)\mathcal{\supset O}(p),\text{~}%
\mathcal{R\supset C}.
\]
Notice also that in general $\mathcal{O\nsupseteqq D}$ unless elements of
$\mathcal{D}$ are of the form $\left[  \left\{  r^{n}\right\}  \right]  $ for
some $r.$

2. We have immediately
\[
\left(  a\left(  x\right)  ,x\right)  \left(  1,b\left(  x\right)  \right)
=\left(  a\left(  x\right)  ,b\left(  x\right)  \right)  .
\]
That is, every Riordan matrix can be decomposed as a product of an Appell and
an associated matrices.
\end{remark}

\begin{proof}
The proof is elementary, based entirely on the Theorem \ref{Roman}.
\end{proof}

\begin{remark}
In \cite{cand12}, it was shown that Appell subgroup is normal in $\mathcal{R}$.
\end{remark}

\begin{remark}
Notice also, following \cite{cand12}, that every Riordan matrix $(g(x),f(x))$
can be presented as the product of an Appell matrix and a Bell one. Namely, we
have
\[
(xg(x)/f(x),x)(f(x)/x,f(x))=(g(x),f(x)).
\]

\end{remark}

\begin{remark}
\label{inverse}Since the matrix $\left(  1,x\right)  $ plays the r\^{o}le of a
neutral element in the Riordan group we see that the inverse of $\left(
p(x),x/(1-\alpha x)\right)  ^{-1}$ is equal to $\left(  p_{1}(x),x/(1+\alpha
x)\right)  $ where $p_{1}(x)$ is such that $p_{1}(x/(1-\alpha x))\allowbreak
=\allowbreak1/p(x).$
\end{remark}

\begin{remark}
All subgroups mentioned on the List of subgroups ...\ref{subgr} except for the
generalized Pascal subgroup are known and mentioned, e.g., in \cite{cand12} or
even earlier, like, e.g., the associated subgroup is mentioned in
\cite{Shap91}.
\end{remark}

Hence, let us analyze subgroup $\mathcal{P}(p,\alpha)$ in more detail.

\begin{proposition}
\label{Pascal}Let us consider a generalized Pascal matrix $\left(
p(x),x/(1-\alpha x)\right)  $. Let us take $\alpha\in\mathbb{R}$ and that
$p\left(  x\right)  $ has the following expansion in the FPS $p(x)\allowbreak
=\allowbreak\sum_{n\geq0}p_{n}x^{n}$ with $p_{0}\neq0$. We have then

1.
\[
\left[  \left[  x^{n}\right]  \right]  p(x)\left(  x/(1-\alpha x)\right)
^{k}=\sum_{j=0}^{n-k}p_{j}\alpha^{n-k-j}\binom{n-j-1}{k-1},
\]

2.
\[
p(x/(1-\alpha x))=p_{0}+\sum_{s=1}^{\infty}x^{s}\sum_{j=0}^{s-1}p_{j+1}%
\alpha^{s-1-j}\binom{s-1}{j}.
\]

\end{proposition}

\begin{proof}
1. Let us find $\left[  \left[  x^{n}\right]  \right]  p(x)\left(  x/(1-\alpha
x)\right)  ^{k}$. To do this, let us expand $\left(  x/(1-\alpha x)\right)
^{k} $ as the power series. It is easy, noticing that for $\alpha
\allowbreak=\allowbreak0$ the series is finite and consists of only one
element, while for $\alpha\neq0$ we have%
\[
\left(  x/(1-\alpha x)\right)  ^{k}=\frac{1}{\alpha^{k}}\left(  \frac{\alpha
x}{1-\alpha x}\right)  ^{k}\allowbreak=\allowbreak\frac{1}{\alpha^{k}%
}(1-\alpha x)\sum_{n\geq k}\binom{n}{k}\left(  \alpha x\right)  ^{n},
\]
by the well-known binomial theorem. Now applying Cauchy rule of multiplication
of series, we get%
\begin{gather*}
p(x)\left(  x/(1-\alpha x)\right)  ^{k}=p\left(  x\right)  \left(  \frac
{1}{\alpha^{k}}(1-\alpha x)\sum_{n\geq0}\binom{n}{k}\left(  \alpha x\right)
^{n}\right) \\
=(1-\alpha x)x^{k}\sum_{s=0}^{\infty}x^{s}\sum_{j=0}^{s}p_{j}\alpha
^{s-j}\binom{s-j+k}{k}.
\end{gather*}

Hence, further we get, after some simple algebra:%
\begin{gather*}
\left[  \left[  x^{n}\right]  \right]  p(x)\left(  x/(1-\alpha x)\right)
^{k}=\sum_{j=0}^{n-k}p_{j}\alpha^{n-k-j}\binom{n-k-j+k}{k}\\
-\alpha\sum_{j=0}^{n-k-1}p_{j}\alpha^{n-k-1-j}\binom{n-k-1-j+k}{k}\\
=\sum_{j=0}^{n-k}p_{j}\alpha^{n-k-j}\binom{n-j-1}{k-1}%
\end{gather*}

2. Let us analyze also what is the formal power series (FPS) of $p(x/(1-\alpha
x))$ provided $FPS$ of $p(x)$ is given by: $\sum_{j\geq0}p_{j}x^{j}.$

We have, after some elementary algebra:%
\begin{gather*}
\sum_{j=0}^{\infty}p_{j}\alpha^{-j}\left(  \frac{\alpha x}{1-\alpha x}\right)
^{j}==(1-\alpha x)\sum_{j=0}^{\infty}p_{j}x^{j}\sum_{k=0}^{\infty}x^{k}%
\alpha^{k}\binom{j+k}{k}\\
=(1-\alpha x)\sum_{s=0}^{\infty}x^{s}\sum_{j=0}^{s}p_{j}\alpha^{s-j}\binom
{s}{s-j}\\
=p_{0}+\sum_{s=1}^{\infty}x^{s}\sum_{j=0}^{s}p_{j}\alpha^{s-j}\binom{s}%
{s-j}-\alpha\sum_{s=0}^{\infty}x^{s+1}\sum_{j=0}^{s}p_{j}\alpha^{s-j}\binom
{s}{s-j}.
\end{gather*}
Further, we get also after some simple algebra:%
\begin{gather*}
p(x/1-\alpha x))=p_{0}+\sum_{s=1}^{\infty}x^{s}\sum_{j=0}^{s}p_{j}\alpha
^{s-j}\binom{s}{j}-\sum_{m=1}^{\infty}x^{m}\sum_{j=0}^{m-1}p_{j}\alpha
^{m-j}\binom{m-1}{j}\\
=p_{0}+\sum_{s=1}^{\infty}x^{s}\sum_{j=0}^{s-1}p_{j+1}\alpha^{s-1-j}%
\binom{s-1}{j}.
\end{gather*}

\end{proof}

We have also the following immediate observations:

\begin{remark}
\label{simple}1. If we set $\alpha\allowbreak=\allowbreak0$ in assertion 1. of
the Proposition \ref{Pascal}, then obviously $\left[  \left[  x^{n}\right]
\right]  p(x)x^{k}\allowbreak=\allowbreak p_{n-k}$. Further, if we set
$\alpha\allowbreak=\allowbreak1\allowbreak=\allowbreak p_{j}$ for all $j\geq0$
then $\left[  \left[  x^{n}\right]  \right]  p(x)\left(  x/(1-x)\right)
^{k}\allowbreak=\allowbreak\sum_{j=0}^{n-k}\binom{n-j-1}{k-1}\allowbreak
=\allowbreak\binom{n}{k}$, i.e., we deal with the so-called Pascal matrix.
That's why we call the Riordan matrix $\left(  p(x),x/(1-\alpha x)\right)  $ a
generalized Pascal matrix.

2. Given an infinite sequence $\left\{  p_{j}\right\}  _{j\geq0}$ the
following infinite sequence \newline$\left\{  \sum_{j=0}^{n}\binom{n}{j}%
p_{j}\alpha^{n-j}\right\}  _{n\geq0}$ is called the sequence of generalized
binomial transforms (GBT($\alpha$)) of the sequence $\left\{  p_{j}\right\}
_{j\geq0}$. In the case of $\alpha\allowbreak=\allowbreak1$, simply the
binomial transform. It is well known that GBT($\alpha$) and GBT($-\alpha$) are
mutually inverse, that is GBT($-\alpha$) applied to GBT($\alpha$) of a
sequence $\left\{  p_{j}\right\}  $ recovers this sequence. Besides GF of the
sequence $\left\{  \hat{p}_{n}\overset{def}{=}\sum_{j=0}^{n}\binom{n}{j}%
p_{j}\alpha^{n-j}\right\}  _{n\geq0}$ is equal to
\[
\frac{1}{1-\alpha x}p(\frac{x}{1-\alpha x}),
\]
hence
\[
p(\frac{x}{1-\alpha x})=(1-\alpha x)\sum_{j\geq0}x^{j}\hat{p}_{j},
\]
and consequently
\[
\lbrack\lbrack x^{n}]]p(\frac{x}{1-\alpha x})=\hat{p}_{n}-\alpha\hat{p}%
_{n-1}.
\]

3. Observe that
\begin{align*}
\left[  \left[  x^{n}\right]  \right]  p(x)\left(  x/(1-\alpha x)\right)
^{k}  &  =\sum_{j=0}^{n-k}p_{j}\alpha^{n-k-j}\binom{n-j-1}{k-1}\\
&  =\binom{n-1}{k-1}\sum_{j=0}^{n-k}\binom{n-k}{j}p_{j}\alpha^{n-k-j}%
/\binom{n-1}{j}%
\end{align*}

\end{remark}

As a corollary, we have the following observations:

\begin{corollary}
\label{Appell}1. Suppose we have two Riordan matrices $(a(x),b(x))$ and
$(d(x),x)$. Then, we have
\[
(a(x),b(x))(d(x),x)\allowbreak=\allowbreak\left(  a(x)d((b(x)),b(x)\right)
\]
and
\[
\left(  d(x),x)\right)  (a(x),b(x))\allowbreak=\allowbreak\left(
d(x)a(x),b(x)\right)  .
\]
In particular, we have the following form of multiplication by the Pascal
matrix:%
\begin{align}
\left(  \frac{1}{1-\alpha x},\frac{x}{1-\alpha x}\right)  (d(x),x) &  =\left(
D(x,\alpha),\frac{x}{1-\alpha x}\right)  ,\label{Mn1}\\
(d(x),x)\left(  \frac{1}{1-\alpha x},\frac{x}{1-x}\right)   &  =\left(
\frac{1}{1-\alpha x}d(x),\frac{x}{1-x}\right)  .\label{Mn2}%
\end{align}
where we denoted $d_{n}\allowbreak=\allowbreak\lbrack\lbrack x^{n}]]d(x),~$
$\hat{d}_{n}\left(  \alpha\right)  \allowbreak=\allowbreak\sum_{j=0}^{n}%
\binom{n}{j}\alpha^{n-j}d_{j}$, \newline$D(x,\alpha)\allowbreak=\allowbreak
\sum_{j=0}^{\infty}\hat{d}_{j}\left(  \alpha\right)  x^{j}.$

2. For $\alpha\in\mathbb{R}$:
\[
\left(  \frac{1}{1-\alpha x},\frac{x}{1-\alpha x}\right)  \left(
d(x),\frac{x}{1+\alpha x}\right)  =\left(  D(x,\alpha),x\right)  ,
\]
where $D(x)$ is defined as above.
\end{corollary}

\begin{proof}
The only thing that requires justification is (\ref{Mn1}). Following Theorem
of Roman, we have%
\[
\left(  \frac{1}{1-x},\frac{x}{1-x}\right)  (d(x),x)=\left(  \frac{1}%
{1-x}d(\frac{x}{1-x}),\frac{x}{1-x}\right)  .
\]
Now, let us recall that
\[
\sum_{n=0}^{\infty}x^{n}\binom{n}{j}\allowbreak=\allowbreak\frac{x^{j}%
}{(1-x)^{j+1}},
\]
consequently we have%
\begin{align*}
d(\frac{x}{1-x})  &  =\sum_{j=0}^{\infty}d_{j}\left(  \frac{x}{1-x}\right)
^{j}=\sum_{j=0}^{\infty}d_{j}(1-x)\sum_{n=j}^{\infty}x^{n}\binom{n}{j}\\
&  =(1-x)\sum_{n=0}^{\infty}x^{n}\sum_{j=0}^{n}d_{j}\binom{n}{j}.
\end{align*}
Now it is elementary to notice that $\frac{1}{1-x}d(\frac{x}{1-x}%
)\allowbreak=\allowbreak\sum_{n=0}^{\infty}x^{n}\hat{d}_{n}\overset{def}{=}%
D(x)$.
\end{proof}

\begin{remark}
As a by-product of assertion 2. of the Lemma \ref{Appell} we see that every
Appell matrix can be decomposed as the product of two generalized Pascal
matrices. Namely, we have for all $\alpha.$%
\begin{align*}
\left(  d(x),\frac{x}{1-\alpha x}\right)  \left(  1,\frac{x}{1+\alpha
x}\right)   &  =\left(  d(x),x\right)  ,\\
\left(  d(x),\frac{x}{1-\alpha x}\right)   &  =\left(  d(x),x\right)  \left(
1,\frac{x}{1-\alpha x}\right)  .
\end{align*}

\end{remark}

\begin{remark}
Since we have
\[
\left(  1,\frac{x}{1-\alpha x}\right)  \left(  1,\frac{x}{1-bx}\right)
=\left(  1,\frac{x}{1-\left(  \alpha+b\right)  x}\right)  ,
\]
as it follows from assertion 2. of the List of subgroups ...\ref{subgr}, we
see that there exists a homomorphism between the group of Riordan matrices
$\left\{  \left(  1,\frac{x}{1-\alpha x}\right)  \right\}  _{a\in\mathbb{C}}$
with matrix multiplication as group operation and the group of complex numbers
with addition as group operation.

Besides notice that
\[
\left[  a^{n-j}\binom{n-1}{j-1}\right]  =\left(  1,\frac{x}{1-\alpha
x}\right)  ,
\]
as it follows from assertion 1. of Proposition \ref{Pascal} with $p_{0}=1$ and
$p_{j}=0$ for $j>0$.
\end{remark}

\begin{remark}
\label{diagon}To complete presentation of simple subgroups of the Riordan
group let us notice that $[\left\{  a^{n}\right\}  ]\allowbreak=\allowbreak
\left(  1,ax\right)  $. Consequently we observe that $\left(  1,\gamma
x\right)  \left(  a(x\right)  ,b(x))\allowbreak=\allowbreak\left(  a\left(
\gamma x\right)  ,b\left(  \gamma x\right)  \right)  ,$ and $\left(
a(x\right)  ,b(x))\left(  1,\gamma x\right)  \allowbreak=\allowbreak\left(
a\left(  x\right)  ,\gamma b\left(  x\right)  \right)  .$
\end{remark}

Why do we study the lower-triangular representation of some sequences? Well,
to get, for example, some nontrivial identities.

To see what we mean let us consider several examples. Most of them would
concern Riordan matrices of the form $\left(  d(x),x\right)  $, i.e., Appell
matrices, since these matrices have a form easy to deal with. Besides, there
exist in the literature many examples of matrices of this form whose entries
are polynomials. Moreover, there are also ready-to-use formulas for the
inverses of such matrices. In the sequel, we often use matrix operation as
defined by (\ref{mnoz}). We we will denote the result of such operation on the
matrix, say $\mathbf{A}$ as $\left[  \left\{  \alpha_{n}\right\}  \right]  $
transform of $\mathbf{A.}$

\section{Examples\label{Ex}}

This section focuses on analyzing well-known identities and sequences of
numbers or polynomials using lower-triangular matrices, with most of them
being of Riordan type. One has to underline that this is not the only approach
to this problem. One can use simply the GF techniques like it was done in
\cite{Barb14} or linear difference equations as it was presented in
\cite{EJ16} and \cite{EJ19}. Most of the examples will be Riordan matrices of
the form $\left(  d(x,y\right)  ,y)$, with $d(x,y)$ being a GF of the sequence
$\left\{  d_{n}(x)\right\}  _{n\geq0}$ consisting either of numbers or
polynomials in $x$. By assertion 1. of the Remark \ref{simple} we see that
$\left[  d_{n-j}(x)\right]  $ is the Riordan matrix $\left(  d(x,y),y\right)
.$ Two examples will refer to the concept of $\mathcal{IP}(p,\beta,\alpha)$
class of Riordan matrices introduced at position (1) of the List of subgroups
...\ref{subgr}. The sub-algebra $\mathcal{SE}$ introduced in the paper's first
section is demonstrated in the subsubsection \ref{Be&Eu}..

\subsection{Bernoulli and Euler polynomials and numbers}

\subsubsection{Bernoulli \& Euler Polynomials}

Let us recall that Bernoulli and Euler polynomials (denoted respectively
$B_{n}(x)$ and $E_{n}(x)$, for $n-$th polynomial) are defined respectively by
the expansions:
\begin{subequations}
\begin{align}
t\exp(xt)/(\exp(t)-1)  &  =\sum_{n\geq0}\frac{t^{n}}{n!}B_{n}(x),\label{Bn}\\
2\exp(xt)/(\exp(t)+1)  &  =\sum_{n\geq0}\frac{t^{n}}{n!}E_{n}(x). \label{En}%
\end{align}
Hence, we have the following identity
\end{subequations}
\begin{gather*}
\sum_{n\geq0}\frac{2^{n}t^{n}}{n!}B_{n}(x)=2t\exp(2xt)/(\exp(2t)-1)=\left(
\sum_{n\geq0}\frac{t^{n}}{n!}B_{n}(x)\right)  \left(  \sum_{n\geq0}\frac
{t^{n}}{n!}E_{n}(x)\right) \\
=\sum_{k\geq0}\frac{t^{k}}{k!}\sum_{n=0}^{k}\binom{k}{n}B_{n}(x)E_{k-n}(x).
\end{gather*}
Comparing polynomials by $t^{n}$ and dividing both sides by $k!$ we end up
with the identity.%
\begin{equation}
2^{k}B_{k}(x)/k!=\sum_{n=0}^{k}\left(  B_{n}(x)/n!\right)  \left(
E_{k-n}(x)/(k-n)!\right)  . \label{B=BE}%
\end{equation}
Let us define the following lower-triangular infinite matrices $\left[
B_{n-j}(x)/(n-j)!\right]  ~$and $\left[  E_{n-j}(x)/(n-j)!\right]  $. Hence,
the equation (\ref{B=BE}) is now reflected in the following relationship
between these matrices:%
\begin{equation}
\left[  B_{n-j}(x)/(n-j)!\right]  \left[  E_{n-j}(x)/(n-j)!\right]  =\left[
\left\{  2^{n}\right\}  \right]  \left[  B_{n-j}(x)/(n-j)!\right]  \left[
\left\{  2^{-j}\right\}  \right]  . \label{B&E}%
\end{equation}

\begin{remark}
Notice that, following definition of Appell matrix, we deduce that $\left[
B_{n-j}(z)/(n-j)!\right]  $ and $\left[  E_{n-j}(z)/(n-j)!\right]  $ are two
Riordan matrix of an Appell type. More precisely, for every $z\in\mathbb{R}:$
\begin{align*}
\left[  B_{n-j}(z)/(n-j)!\right]  \allowbreak &  =\allowbreak(\frac{x\exp
(zx)}{\exp(x)-1},x),\\
\left[  E_{n-j}(z)/(n-j)!\right]   &  =(\frac{2\exp(zx)}{\exp(x)+1},x)
\end{align*}
. Moreover, (\ref{B&E}) is just the example of the position (3). of the List
of subgroups ...\ref{subgr}, above. Namely, we have
\begin{align*}
&  \left[  B_{n-j}(z)/(n-j)!\right]  \left[  E_{n-j}(z)/(n-j)!\right]  =\\
&  =(\frac{x\exp(zx)}{\exp(x)-1},x)(\frac{2\exp(zx)}{\exp(x)+1},x)=\left(
\frac{2\exp(2zx)}{\exp(2x)-1},x\right)  ,
\end{align*}
where $\frac{2\exp(2zx)}{\exp(2x)-1}\allowbreak=\allowbreak\sum_{n>j}%
x^{n-j}2^{n-j}B_{n-j}(z)/(n-j)!$.
\end{remark}

\subsubsection{Laguerre polynomials}

Let $L_{n}^{\left(  \alpha\right)  }(x)$ denote $n-$th Laguerre polynomial,
the member of the family of polynomials which are orthogonal with respect to
measure $x^{-\alpha}\exp(-x)$ for $x>0$ and $\alpha\neq1,2,\ldots$ . Let us
define the following family of orthogonal polynomials:%
\[
\mathcal{L}_{n}^{\left(  \alpha\right)  }(x)=L_{n}^{\left(  -\alpha-2\right)
}\left(  -x)\right)  .
\]

\begin{proposition}
For all $\alpha\neq1,2,\ldots$ ,$n\geq m\geq0:\sum_{j=m}^{n}L_{n-j}^{\left(
\alpha\right)  }(x)\mathcal{L}_{j-m}^{\left(  \alpha\right)  }(x)=0$,
Consequently, we have
\[
\left[  L_{n-j}^{\left(  \alpha\right)  }(x)\right]  ^{-1}=\left[
\mathcal{L}_{n-j}^{\left(  \alpha\right)  }(x)\right]  .
\]
Besides, we know that we deal with the Riordan matrix of the Appell type
\[
\left[  L_{n-j}^{(a)}(x)\right]  =\left(  \exp\left(  -\frac{tx}{1-t}\right)
/\left(  1-t\right)  ^{a+1},1\right)  ,
\]
since $\sum_{j\geq0}t^{n}L_{n}(x)\allowbreak=\allowbreak\exp\left(  -\frac
{tx}{1-t}\right)  /\left(  1-t\right)  .$
\end{proposition}

\begin{proof}
First, notice that the we can take $k\allowbreak=\allowbreak j-m$ and prove
the identity $\sum_{k=0}^{n-m}L_{n-m-k}^{\left(  a\right)  }(x)\mathcal{L}%
_{k}^{\left(  a\right)  }(x)\allowbreak=\allowbreak0$. We are now using the
concept of generating functions and asserting Remark \ref{inverse}.
Consequently, we realize that our identity is equivalent to the fact that the
generating function of the polynomials $\left\{  \mathcal{L}_{k}^{\left(
\alpha\right)  }(x)\right\}  $ is equal to the inverse of the generating
function of the Laguerre polynomials. But it is well-known that
\[
\sum_{n\geq0}t^{n}L_{n}^{\left(  \alpha\right)  }(x)=\frac{1}{\left(
1-t\right)  ^{\alpha+1}}\exp(-\frac{tx}{1-t}).
\]
Hence, it remains to notice that :
\[
\sum_{n\geq0}t^{n}\mathcal{L}_{n}^{\left(  \alpha\right)  }(x)=(1-t)^{\alpha
+1}\exp(\frac{tx}{1-t}).
\]

\end{proof}

\subsubsection{Hermite polynomials}

Following an identity presented in, say \cite{Szab2020}(unnumbered formula
below (8.17)), we have%
\[
\left[  He_{n-j}(x)/(n-j)!\right]  ^{-1}=\left[  i^{n-j}He_{n-j}%
(ix)/(n-j)!\right]  ,
\]
where $i$ is an imaginary unit, and $He_{n}(x)$ $n-$th so-called probabilistic
Hermite polynomial, i.e., the ones orthogonal with respect to the measure with
the density $\exp(-x^{2}/2)/\sqrt{2\pi}$.

Recall also that the GF function of the Hermite polynomials is $\exp
(yx-y^{2}/2)$ that is we have (see e.g. Wikipedia, \cite{KLS})%
\[
\sum_{n\geq0}\frac{y^{n}}{n!}He_{n}(x)=\exp(yx-y^{2}/2).
\]

\begin{remark}
Now we see that $\left[  He_{n-j}(x)/(n-j)!\right]  $ is an Riordan matrix of
the Appell type given by the formula $\left(  \exp(yx-y^{2}/2),x\right)  $.
\end{remark}

\subsubsection{Numbers\label{Be&Eu}}

We will also use the following notation:
\begin{align*}
\varepsilon(n)  &  =\left\{
\begin{array}
[c]{ccc}%
0 & \text{if } & n\text{ is odd}\\
1 & \text{if} & n\text{ is even}%
\end{array}
\right.  ,~H(x)=\left\{
\begin{array}
[c]{ccc}%
0 & \text{if} & x<0\\
1 & \text{if} & x\geq0
\end{array}
\right.  ,\\
H2(x)  &  =H(x)+H(x-2)-2H(x-1)=\left\{
\begin{array}
[c]{ccc}%
1 & \text{if} & x=0\\
-1 & \text{if} & x=1\\
0 & \text{if} & \text{otherwise}%
\end{array}
\right.  ,\\
H3(x)  &  =H(x)-H(x-1)-H(x-2)+H(x-3).
\end{align*}
and $\left\{  E_{n}\right\}  $, and $\left\{  B_{n}\right\}  $ denote
respectively $n-$th Euler and Bernoulli numbers.

We have
\begin{align*}
\left[  \binom{n}{j}\right]  ^{-1} &  =\left[  (-1)^{n-j}\binom{n}{j}\right]
,~\left[  \lambda^{n-j}\binom{n}{j}\right]  ^{-1}=\left[  (-\lambda
)^{n-j}\binom{j}{n}\right]  ,\\
\left[  H(n-j)\right]  ^{-1} &  =\left[  H2(n-j)\right]  ,~\left[
\varepsilon(n-j)H(n-j)\right]  ^{-1}=\left[  H3(n-j)\right]  ,
\end{align*}
for some real $\lambda$. Let us recall that matrix $\left[  \binom{n}%
{j}\right]  $ it is a well know Pascal matrix mentioned already above. Hence,
we have for example, for all $n>j$
\[
\sum_{k=j}^{n}(-1)^{k}\binom{n}{k}\binom{k}{j}=0.
\]

\begin{remark}
Notice also that $\left[  H[n-j\right]  \allowbreak=\allowbreak\left(
1/(1-x),x\right)  $, hence naturally we have $\left[  H\left(  n-j\right)
\right]  ^{-1}\allowbreak=\allowbreak\left(  1-x,x\right)  $, as it follows
from the List of subgroups ...\ref{subgr}, 3. .
\end{remark}

2. As shown in \cite{Szabl14} we have%
\begin{equation}
\left[  \binom{n}{j}\frac{1}{n-j+1}\right]  ^{-1}=\left[  \binom{n}{j}%
B_{n-j}\right]  ,~\left[  \varepsilon(n-j)\binom{n}{j}\right]  ^{-1}=\left[
\binom{n}{j}E_{n-j}\right]  . \label{BiE}%
\end{equation}

Thus we have for all $n>j:$%
\begin{align*}
\sum_{k=j}^{n}\binom{n}{k}\binom{k}{j}\frac{B_{k-j}}{n-k+1}  &  =0,\\
\sum_{k=j}^{n}e(n-k)\binom{n}{k}\binom{k}{j}E_{k-j}  &  =0.
\end{align*}

Notice that performing the following matrix multiplications on the first of
the identities (\ref{BiE}):%
\[
\left[  \left\{  1/n!\right\}  \right]  \left[  \binom{n}{j}B_{n-j}\right]
\left[  \left\{  j!\right\}  \right]  ,~\left[  \left\{  1/n!\right\}
\right]  \left[  \binom{n}{j}\frac{1}{n-j+1}\right]  ^{-1}\left[  \left\{
j!\right\}  \right]  ,
\]
we end up with the rather unexpected identities:%
\begin{align}
\left[  1/(n-j)!\right]  ^{-1}  &  =\left[  (-1)^{n-j}/(n-j)!\right]
\label{sila}\\
\left[  1/(n-j+1)!\right]  ^{-1}  &  =\left[  B_{n-j}/(n-j)!\right]
,\label{silb}\\
\left[  \varepsilon(j-i)/(j-i)!\right]  ^{-1}  &  =\left[  E_{j-i}%
/(j-i)!\right]  . \label{silc}%
\end{align}

\begin{remark}
Let us notice that $\left[  1/(n-j)!\right]  $, $\left[  1/(n-j+1)!\right]
,\left[  \varepsilon(j-i)/(j-i)!\right]  $, $\left[  B_{n-j}/(n-j)!\right]  $,
$\left[  E_{j-i}/(j-i)!\right]  $ are Riordan matrices and $\left[
1/(n-j)!\right]  \allowbreak=\allowbreak\left(  \exp(x),x\right)  $, $\left[
1/(n-j+1)!\right]  \allowbreak=\allowbreak\left(  \exp(x\right)  -1,x)$,
$\left[  \varepsilon(j-i)/(j-i)!\right]  \allowbreak=\allowbreak\left(
\cosh(x),x\right)  $, \newline$\left[  B_{n-j}/(n-j)!\right]  \allowbreak
=\allowbreak\left(  x\exp(x)/(\exp(x)-1),x\right)  $, $\left[  E_{j-i}%
/(j-i)!\right]  \allowbreak=\allowbreak\left(  2\exp(x)/(\exp(x)+1),x\right)
.$ Besides we see that $\left[  \binom{n}{j}E_{n-j}\right]  ,$ $\left[
\varepsilon(n-j)\binom{n}{j}\right]  ,$ $\left[  \varepsilon
(j-i)/(j-i)!\right]  ~,\left[  E_{j-i}/(j-i)!\right]  $ are the examples of
the elements of the sub-algebra $\mathcal{SE}$.
\end{remark}

3. As shown also in \cite{Szabl14} we have further%
\begin{align*}
\left[  \varepsilon(n-j)\binom{n}{j}\frac{1}{n-j+1}\right]  ^{-1}  &  =\left[
\binom{n}{j}\sum_{k=0}^{n-j}\binom{n-j}{k}2^{k}B_{k}\right]  ,\\
\left[  \binom{2n}{2j}\right]  ^{-1}  &  =\left[  \binom{2n}{2j}%
E_{2(n-j)}\right]  ,\\
~~\left[  \binom{2n}{2j}\frac{1}{2(n-j)+1}\right]  ^{-1}  &  =\left[
\binom{2n}{2j}\sum_{k=0}^{2(n-j)}\binom{2n-2j}{k}2^{k}B_{k}\right]  .
\end{align*}

Hence, in particular we have $\forall s\geq1:$%
\begin{gather*}
\sum_{k=0}^{s}\binom{2s}{2k}E_{2k}=0,~\\
\sum_{k=0}^{2s}\binom{2s}{k}B_{k}/(2s-k+1)=0.
\end{gather*}
Notice that we have changed the order of summation to get the second identity
and use the following, elementary to prove, identity:%
\[
\sum_{m=\left\lfloor l/2\right\rfloor }^{s}\binom{2s-l}{2m-l}/(2s-2m+1)=\frac
{2^{2s-l}}{2s-l+1}.
\]

Again, performing right hand side and left-hand side multiplications by
matrices $\left[  \left\{  1/n!\right\}  \right]  $ and its inverse we get the
following also unexpected identities:%
\begin{align}
\left[  \varepsilon(j-i)/(j-i+1)!\right]  ^{-1}  &  =\left[  \sum_{k=0}%
^{j-i}2^{k}B_{j-i-k}/(k!(j-i-k)!)\right]  ,\label{sil2a}\\
\left[  1/(2j-2i)!\right]  ^{-1}  &  =\left[  E_{2j-2i}/(2j-2i)!\right]
,\label{sil2b}\\
\left[  1/(2j-2i+1)!\right]  ^{-1}  &  =\left[  \sum_{k=0}^{2j-2i}2^{k}%
B_{k}/(k!(2j-2i-k)!)\right]  . \label{sil2c}%
\end{align}

\begin{remark}
Again notice that $\left[  \varepsilon(n-j)/(n-j+1)!\right]  $, $\left[
1/(2n-2j)!\right]  ,\left[  1/(2n-2i+1)!\right]  $, $\left[  E_{2j-2i}%
/(2j-2i)!\right]  $ are Riordan matrices respectively: $\allowbreak\left(
\sinh(x)/x,x\right)  $, $\allowbreak(\cosh(\sqrt{x}),x)$, $\left(  \sinh
(\sqrt{x}\right)  /\sqrt{x},x)\allowbreak$, $\left(  \sqrt{x}/\sinh\left(
\sqrt{x}\right)  ,x\right)  $. \newline Consequently we deduce that $\left[
\sum_{k=0}^{j-i}2^{k}B_{j-i-k}/(k!(j-i-k)!)\right]  $, \newline$\left[
\sum_{k=0}^{2j-2i}2^{k}B_{k}/(k!(2j-2i-k)!)\right]  $ are also Riordan
matrices respectively: $\left(  x/\sinh\left(  x\right)  ,x\right)  $ and
$\left(  \sqrt{x}/\sinh\left(  x\right)  ,x\right)  $. \newline Again notice
that the lower-triangular infinite matrices $\left[  \varepsilon(n-j)\binom
{n}{j}\frac{1}{n-j+1}\right]  $ and $\left[  \varepsilon(j-i)/(j-i+1)!\right]
$ belong to the sub-algebra $\mathcal{SE}$.
\end{remark}

\subsection{Pochhammer symbol - rising factorials}

We will use the following notation. For $x\in\mathbb{C}$ let us denote%
\begin{equation}
(x)_{(n)}\allowbreak=\allowbreak x(x-1)\ldots(x-n+1). \label{down}%
\end{equation}
This polynomial in $x$ will be called falling factorial while the following
polynomial
\begin{equation}
(x)^{(n)}=x(x+1)\ldots(x+n-1), \label{up}%
\end{equation}
will be called rising factorial. The $n-$th rising factorial is also called
the Pochhammer symbol. In both cases, we set $1$ when $n\allowbreak
=\allowbreak0$.

It is also well-know that
\begin{equation}
(x)_{(n)}=(-1)^{n}(-x)^{(n)},\text{and }(x)^{(n)}=(-1)^{n}(-x)_{(n)}.
\label{Pochh}%
\end{equation}
Recall, that we have also the so-called binomial theorem stating that for all
complex $\left\vert x\right\vert <1$ we have%
\begin{equation}
(1-x)^{\alpha}\allowbreak=\allowbreak\sum_{j\geq0}(-x)^{j}(\alpha
)_{(j)}/j!=\sum_{j\geq0}x^{j}(-\alpha)^{(j)}/j!. \label{bin}%
\end{equation}

Notice that from this expansion, by the (so-called exponential) generating
function method we get the following identity which is true for all
$\alpha,\beta\in\mathbb{C}$:%
\begin{align}
\left(  \alpha+\beta\right)  ^{(n)}  &  =\sum_{j=0}^{n}\binom{n}{j}\left(
\alpha\right)  ^{(n-j)}\left(  \beta\right)  ^{\left(  j\right)
},\label{sum1}\\
\left(  \alpha+\beta\right)  _{\left(  n\right)  }  &  =\sum_{j=0}^{n}%
\binom{n}{j}\left(  \alpha\right)  _{(n-j)}\left(  \beta\right)  _{\left(
j\right)  }. \label{sum2}%
\end{align}

The following is our initial observation:

\begin{proposition}
\label{exp}i) $\forall\mathbb{C}\ni x\neq0:$%
\begin{equation}
\left[  (-1)^{j}\binom{n}{j}\frac{\left(  x\right)  ^{(n)}}{\left(  x\right)
^{\left(  j\right)  }}\right]  ^{-1}=\left[  (-1)^{j}\binom{n}{j}\frac{\left(
x\right)  ^{(n)}}{\left(  x\right)  ^{\left(  j\right)  }}\right]  .
\label{=odwr}%
\end{equation}

ii) $\forall\mathbb{C}\ni x\neq0:$%
\begin{equation}
\left[  \binom{n}{j}\frac{\left(  x\right)  ^{(n)}}{\left(  x\right)
^{\left(  j\right)  }}\right]  ^{-1}=\left[  (-1)^{n-j}\binom{n}{j}%
\frac{\left(  x\right)  ^{(n)}}{\left(  x\right)  ^{\left(  j\right)  }%
}\right]  . \label{=odwrii}%
\end{equation}

iii) For the reference sequence $c_{n}\allowbreak=\allowbreak n!$, we have%
\[
\left[  (-1)^{j}\binom{n}{j}\frac{\left(  x\right)  ^{(n)}}{\left(  x\right)
^{\left(  j\right)  }}\right]  \allowbreak=\allowbreak\left(  \left(
1-y\right)  ^{-x},\frac{-y}{1-y}\right)  .
\]
Thus, it belongs to a class $\mathcal{IP}(p,-1,1)$ (compare position (1) of
the List of subgroups ...\ref{subgr}.
\end{proposition}

\begin{proof}
To prove this identity, we have to show that $\forall n>j:$%
\begin{align*}
0  &  =\sum_{k=j}^{n}(-1)^{k}\binom{n}{k}\frac{\left(  x\right)  ^{(n)}%
}{\left(  x\right)  ^{\left(  k\right)  }}(-1)^{j}\binom{k}{j}\frac{\left(
x\right)  ^{(k)}}{\left(  x\right)  ^{\left(  j\right)  }}\\
&  =\binom{n}{j}\frac{\left(  x\right)  ^{(n)}}{\left(  x\right)  ^{\left(
j\right)  }}\sum_{k=j}^{n}(-1)^{k-j}\frac{(n-j)!}{(n-k)!(k-j)!}.
\end{align*}
But this is obvious in the face of the above calculations.

ii) We apply right-hand side multiplication of a lower triangular infinite
matrix by a diagonal matrix $\left\{  (-1)^{j}\right\}  $ as described in
Remark \ref{proste},4. getting on one hand $\left[  (-1)^{j}\binom{n}{j}%
\frac{\left(  x\right)  ^{(n)}}{\left(  x\right)  ^{\left(  j\right)  }%
}\right]  [\left\{  \left(  -1\right)  ^{j}\right\}  ]\allowbreak
=\allowbreak\left[  \binom{n}{j}\frac{\left(  x\right)  ^{(n)}}{\left(
x\right)  ^{\left(  j\right)  }}\right]  .$ On the other hand since
\begin{align*}
&  \left(  \left[  (-1)^{j}\binom{n}{j}\frac{\left(  x\right)  ^{(n)}}{\left(
x\right)  ^{\left(  j\right)  }}\right]  [\left\{  \left(  -1\right)
^{j}\right\}  ]\right)  ^{-1}\\
&  =[\left\{  \left(  -1\right)  ^{j}\right\}  ]^{-1}\left[  (-1)^{j}\binom
{n}{j}\frac{\left(  x\right)  ^{(n)}}{\left(  x\right)  ^{\left(  j\right)  }%
}\right] \\
&  =[\left\{  \left(  -1\right)  ^{n}\right\}  ]\left[  (-1)^{j}\binom{n}%
{j}\frac{\left(  x\right)  ^{(n)}}{\left(  x\right)  ^{\left(  j\right)  }%
}\right]  =\left[  (-1)^{n-j}\binom{n}{j}\frac{\left(  x\right)  ^{(n)}%
}{\left(  x\right)  ^{\left(  j\right)  }}\right]
\end{align*}

iii) Following (\ref{bin}), we get for the $j-$th column
\begin{align*}
g_{j}(x,y)  &  =\sum_{n\geq j}(-1)^{j}\binom{n}{j}\frac{\left(  x\right)
^{\left(  n\right)  }y^{n}}{\left(  x\right)  ^{\left(  j\right)  }%
n!}=(-1)^{j}\frac{y^{j}}{j!}\sum_{n\geq j}y^{n-j}\frac{\left(  x+j\right)
^{\left(  n-j\right)  }}{(n-j)!}\\
&  =(-1)^{j}\frac{y^{j}}{j!}(1-y)^{-x-j}=(1-y)^{-x}(-\frac{y}{1-y})^{j}/j!.
\end{align*}

\end{proof}

Now, taking $\alpha\allowbreak=\allowbreak x$ and $\beta\allowbreak
=\allowbreak-x$ in (\ref{sum1}) and then in (\ref{sum2}) and finally using the
standard trick with multiplication by the diagonal matrix $\left[  \left\{
n!\right\}  \right]  $ and its inverse (compare Remark \ref{proste},4.), we
arrive at the following identities%
\begin{align*}
\left[  \left(  x\right)  ^{\left(  n-j\right)  }/(n-j)!\right]  ^{-1} &
=\left[  \left(  -x\right)  ^{\left(  n-j\right)  }/(n-j)!\right]  ,\\
\left[  \left(  x\right)  _{\left(  n-j\right)  }/\left(  n-j\right)
!\right]  ^{-1} &  =\left[  \left(  -x\right)  _{\left(  n-j\right)
}/(n-j)!\right]  .
\end{align*}

Similarly, setting $x\allowbreak=\allowbreak1$ in (\ref{=odwr},i)), we get%
\[
\left[  (-1)^{j}(n-j)!\binom{n}{j}^{2}\right]  ^{-1}=\left[  (-1)^{j}%
(n-j)!\binom{n}{j}^{2}\right]  .
\]

Similarly, setting $x\allowbreak=\allowbreak1$ in (\ref{=odwrii}), we get%
\[
\left[  (n-j)!\binom{n}{j}^{2}\right]  ^{-1}=\left[  (-1)^{n-j}(n-j)!\binom
{n}{j}^{2}\right]  ,
\]

\begin{remark}
It is elementary to notice that $\left[  \left(  x\right)  ^{\left(
n-j\right)  }/(n-j)!\right]  $ is equal to the Riordan matrix of Appell
type$((1-y)^{-x},y).$ Further let us note, that $\left[  (-1)^{j}%
(n-j)!\binom{n}{j}^{2}\right]  $ is a Riordan matrix $(\frac{1}{1-y},\frac
{-y}{1-y}).$ Finally, combining the fact that $\left[  (-1)^{j}(n-j)!\binom
{n}{j}^{2}\right]  $ is a Riordan matrix $(\frac{1}{1-y},\frac{-y}{1-y}),$ the
fact that
\[
\left[  (n-j)!\binom{n}{j}^{2}\right]  \allowbreak=\allowbreak\left[
(-1)^{j}(n-j)!\binom{n}{j}^{2}\right]  [\left\{  (-1)^{j}\right\}  ],
\]
and the observations contained in the Remark \ref{diagon}, we deduce the
$\left[  (n-j)!\binom{n}{j}^{2}\right]  $ is the Riordan matrix of the Pascal
type equal to $(\frac{1}{1-y},\frac{y}{1-y}).$
\end{remark}

Recently, in \cite{Szabl24} several lower-triangular matrices involving rising
factorials of two variables have been defined. Let us present these matrices
and some relationships they are involved in. Let us remark that the matrices
presented below are not Riordan matrices. However, since they have rather
complicated structure, they can be used to obtain non-trivial relationships
involving Pochhammer symbols of two variables.

1.
\begin{align*}
E(a,b)  &  =\left[  \frac{\left(  a+b+n-1\right)  ^{\left(  j\right)  }}%
{j!}\frac{(b+j)^{(n-j)}}{(n-j)!}\right]  ^{-1}\\
\bar{E}(a,b)  &  =E^{-1}(a,b)=\left[  (-1)^{n-j}\frac{n!\left(  b+j\right)
^{\left(  n-j\right)  }\left(  a+b+2j-1\right)  }{(n-j)!\left(
a+b+j-1\right)  ^{\left(  n+1\right)  }}\right]  .
\end{align*}

2.%
\begin{align*}
E(a,b)\bar{E}(b,b)  &  =\left[  (-1)^{n-j}\frac{(b+j)^{(n-j)}(a-b)^{(n-j)}%
(a+b+n-1)^{(j)}(2b+2j-1)}{(n-j)!(2b+j-1)^{(n+1)}}\right]  ,\\
E(b,a)\bar{E}\left(  b,b\right)   &  =\left[  (-1)^{n}\right]  E(a,b)\bar
{E}(b,b)\left[  \left(  -1\right)  ^{j}\right] \\
&  =\left[  \frac{(b+j)^{(n-j)}(a-b)^{(n-j)}(a+b+n-1)^{(j)}(2b+2j-1)}%
{(n-j)!(2b+j-1)^{(n+1)}}\right]  ,\\
E(b,b)\bar{E}(a,b)  &  =\left(  E(a,b)\bar{E}(b,b)\right)  ^{-1}\\
&  =\left[  (-1)^{n-j}\frac{(b+j)^{(n-j)}(2b+n-1)^{(j)}(b-a)^{(n-j)}\left(
a+b+2j-1\right)  }{(n-j)!(a+b+j-1)^{(n+1)}}\right]  ,\\
E\left(  b,b\right)  \bar{E}(b,a)  &  =\left[  (-1)^{n}\right]  E(b,b)\bar
{E}(a,b)\left[  \left(  -1\right)  ^{j}\right] \\
&  =\left[  \frac{(b+j)^{(n-j)}(2b+n-1)^{(j)}(b-a)^{(n-j)}\left(
a+b+2j-1\right)  }{(n-j)!(a+b+j-1)^{(n+1)}}\right]  .
\end{align*}

3.%
\begin{gather*}
E\left(  a,a\right)  \bar{E}\left(  b,b\right)  =\\
\left[  \frac{\varepsilon\left(  n-j\right)  (2b+2j-1)\left(  2a+n-1\right)
^{\left(  j\right)  }\left(  a-b\right)  ^{\left(  \left(  n-j\right)
/2\right)  }\left(  b+j\right)  ^{\left(  \left(  n-j\right)  /2\right)
}\left(  a+\left(  n+j\right)  /2\right)  ^{\left(  \left(  n-j\right)
/2\right)  }}{\left(  \left(  n-j\right)  /2\right)  !\left(  2b+j-1\right)
^{\left(  n+1\right)  }}\right]  .
\end{gather*}

\subsection{Examples coming from $q-$series theory}

\subsubsection{Introduction and notation}

Let us introduce a few elementary notions of the so-called $q-$series theory.

$q$ is a parameter. Generally, we will assume that $-1<q<1,$ unless otherwise
stated. However, sometimes we will consider the case $q\allowbreak
=\allowbreak1$, not always directly, but as a left-hand side limit (
i.e.,$q\longrightarrow1^{-}$). We will point out these cases.

We will use traditional notations of the $q-$series theory i.e.,%
\[
\left[  0\right]  _{q}\allowbreak=\allowbreak0,~\left[  n\right]
_{q}\allowbreak=\allowbreak1+q+\ldots+q^{n-1}\allowbreak,\left[  n\right]
_{q}!\allowbreak=\allowbreak\prod_{j=1}^{n}\left[  j\right]  _{q},\text{with
}\left[  0\right]  _{q}!\allowbreak=1\text{,}%
\]%
\[%
\genfrac{[}{]}{0pt}{}{n}{k}%
_{q}=\left\{
\begin{array}
[c]{ccc}%
\frac{\left[  n\right]  _{q}!}{\left[  n-k\right]  _{q}!\left[  k\right]
_{q}!} & , & n\geq k\geq0\\
0 & , & \text{otherwise}%
\end{array}
\right.  \text{.}%
\]
$\binom{n}{k}$ will denote the ordinary, well known binomial coefficient. It
turns out that $%
\genfrac{[}{]}{0pt}{}{n}{k}%
_{q}$ are polynomials in $q$ (called Gauss polynomials).

It is useful to use the so-called $q-$Pochhammer symbol for $n\geq1:$%
\begin{equation}
\left(  a|q\right)  _{n}=\prod_{j=0}^{n-1}\left(  1-aq^{j}\right)  ,~~\left(
a_{1},a_{2},\ldots,a_{k}|q\right)  _{n}\allowbreak=\allowbreak\prod_{j=1}%
^{k}\left(  a_{j}|q\right)  _{n}\text{,} \label{qP}%
\end{equation}
with $\left(  a|q\right)  _{0}\allowbreak=\allowbreak1$.

Although the formula below was known much earlier we cite \cite{Pol71} because
of its nice proof. Namely, we have%
\begin{equation}
\left(  a|q\right)  _{n}=\sum_{j=0}^{n}(-a)^{j}q^{\binom{j}{2}}%
\genfrac{[}{]}{0pt}{}{n}{j}%
_{q}. \label{fbin}%
\end{equation}
This formula is often referred to as a finite $q-$binomial formula, that will
be generalized just below.

Often $\left(  a|q\right)  _{n}$ as well as $\left(  a_{1},a_{2},\ldots
,a_{k}|q\right)  _{n}$ will be abbreviated to $\left(  a\right)  _{n}$ (not to
be confused with the falling factorial defined above) and \newline$\left(
a_{1},a_{2},\ldots,a_{k}\right)  _{n}$, if it will not cause misunderstanding.

We will also use the following symbol $\left\lfloor n\right\rfloor $ to denote
the largest integer not exceeding $n$.

It is worth to mention the following two formulae, that are well known.
Namely, the following formulae are true for $\left\vert t\right\vert <1$,
$\left\vert q\right\vert <1$ (derived already by Euler, see \cite{Andrews1999}
Corollary 10.2.2)
\begin{align}
\frac{1}{(t)_{\infty}}\allowbreak &  =\allowbreak\sum_{k\geq0}\frac{t^{k}%
}{(q)_{k}}\text{,}\label{binT}\\
(t)_{\infty}\allowbreak &  =\allowbreak\sum_{k\geq0}(-1)^{k}q^{\binom{k}{2}%
}\frac{t^{k}}{(q)_{k}}\text{.} \label{obinT}%
\end{align}

It is easy to notice that $\left(  q\right)  _{n}=\left(  1-q\right)
^{n}\left[  n\right]  _{q}!$ and that%
\begin{equation}%
\genfrac{[}{]}{0pt}{}{n}{k}%
_{q}\allowbreak=\allowbreak\allowbreak\left\{
\begin{array}
[c]{ccc}%
\frac{\left(  q\right)  _{n}}{\left(  q\right)  _{n-k}\left(  q\right)  _{k}}
& , & n\geq k\geq0\\
0 & , & \text{otherwise}%
\end{array}
\right.  \text{.} \label{dqbin}%
\end{equation}
\newline The above-mentioned formula is just an example where direct setting
$q\allowbreak=\allowbreak1$ is senseless however, the passage to the limit
$q\longrightarrow1^{-}$ makes sense.

Notice that, in particular,
\begin{equation}
\left[  n\right]  _{1}\allowbreak=\allowbreak n,~\left[  n\right]
_{1}!\allowbreak=\allowbreak n!,~%
\genfrac{[}{]}{0pt}{}{n}{k}%
_{1}\allowbreak=\allowbreak\binom{n}{k},~(a)_{1}\allowbreak=\allowbreak
1-a,~\left(  a|1\right)  _{n}\allowbreak=\allowbreak\left(  1-a\right)  ^{n}
\label{q1}%
\end{equation}
and
\begin{equation}
\left[  n\right]  _{0}\allowbreak=\allowbreak\left\{
\begin{array}
[c]{ccc}%
1 & \text{if} & n\geq1\\
0 & \text{if} & n=0
\end{array}
\right.  ,~\left[  n\right]  _{0}!\allowbreak=\allowbreak1,~%
\genfrac{[}{]}{0pt}{}{n}{k}%
_{0}\allowbreak=\allowbreak1,~\left(  a|0\right)  _{n}\allowbreak
=\allowbreak\left\{
\begin{array}
[c]{ccc}%
1 & \text{if} & n=0\\
1-a & \text{if} & n\geq1
\end{array}
\right.  . \label{q2}%
\end{equation}

\subsubsection{Examples concerning important in $q-$series theory families of
polynomials}

To proceed further, let us prove some auxiliary results.

\begin{proposition}
\label{aux}1. $\forall x,y,q\in\mathbb{C},\left\vert q\right\vert ,\left\vert
y\right\vert <1:$%
\[
\sum_{j=0}^{\infty}\frac{y^{j}}{\left(  q\right)  _{j}}\left(  x\right)
_{j}=\frac{\left(  yx\right)  _{\infty}}{\left(  y\right)  _{\infty}},
\]

2. $\forall x,q\in\mathbb{C},x\neq0,\left\vert q\right\vert <1,n\geq1$%
\begin{equation}
\sum_{j=0}^{n}%
\genfrac{[}{]}{0pt}{}{n}{j}%
_{q}\left(  x\right)  _{n-j}x^{j}\left(  1/x\right)  _{j}=0. \label{qnaw}%
\end{equation}

\end{proposition}

\begin{proof}
1. Using (\ref{fbin}), we get
\begin{align*}
\sum_{j=0}^{\infty}\frac{y^{j}}{\left(  q\right)  _{j}}\left(  x\right)  _{j}
&  =\sum_{j=0}^{\infty}\frac{y^{j}}{\left(  q\right)  _{j}}\sum_{k=0}^{j}%
\genfrac{[}{]}{0pt}{}{j}{k}%
_{q}q^{\binom{k}{2}}(-x)^{k}\\
&  =\sum_{k=0}^{\infty}\frac{(yx)^{k}}{\left(  q\right)  _{k}}(-1)^{k}%
q^{\binom{k}{2}}\sum_{j=k}^{\infty}\frac{y^{j-k}}{\left(  q\right)  _{j-k}%
}=\frac{\left(  yx\right)  _{\infty}}{\left(  y\right)  _{\infty}},
\end{align*}
by (\ref{binT}) and (\ref{obinT}).

2. Now notice that using assertion 1. with $y$ substituted by $xy$ we get
\[
\frac{\left(  y\right)  _{\infty}}{\left(  yx\right)  _{\infty}}=\frac{\left(
yx/x\right)  _{\infty}}{\left(  yx\right)  _{\infty}}=\sum_{j=0}^{\infty}%
\frac{y^{j}}{\left(  q\right)  _{j}}x^{j}\left(  1/x\right)  _{j}.
\]

Let us denote $d_{n}(x|q)\allowbreak=\allowbreak\sum_{j=0}^{n}%
\genfrac{[}{]}{0pt}{}{n}{j}%
_{q}\left(  x\right)  _{n-j}x^{j}\left(  1/x\right)  _{j}$ and let us find its
generating function. We have%
\begin{gather*}
\sum_{n=0}^{\infty}\frac{y^{n}}{\left(  q\right)  _{n}}d_{n}(x|q)=\sum
_{n=0}^{\infty}y^{n}\sum_{j=0}^{n}\frac{\left(  x\right)  _{n-j}}{\left(
q\right)  _{n-j}}\frac{x^{j}\left(  1/x\right)  _{j}}{\left(  q\right)  _{j}%
}=\\
\sum_{j=0}^{\infty}\frac{y^{j}x^{j}\left(  1/x\right)  _{j}}{\left(  q\right)
_{j}}\sum_{n=j}^{\infty}\frac{y^{n-j}}{\left(  q\right)  _{n-j}}\left(
x\right)  _{n-j}=\frac{\left(  y\right)  _{\infty}}{\left(  yx\right)
_{\infty}}\frac{\left(  yx\right)  _{\infty}}{\left(  y\right)  _{\infty}}=1.
\end{gather*}
Hence $d_{n}(x|q)\allowbreak=\allowbreak0$ for $n\geq1.$
\end{proof}

We also have the following almost elementary observation:

Following (\ref{fbin}) for $n>0$ and $a\allowbreak=\allowbreak1$ we have he
following identity, true for $n\geq1:$%
\[
\sum_{j=0}^{n}%
\genfrac{[}{]}{0pt}{}{n}{j}%
_{q}(-1)^{j}q^{\binom{j}{2}}=0.
\]

As a corollary we have the following relationship:

\begin{corollary}
\label{vqbin}1. $\forall x,q\in\mathbb{C}$,$\left\vert q\right\vert \neq1:$%
\[
\left[  x^{n-j}/(q)_{n-j}\right]  ^{-1}\allowbreak=\allowbreak\left[
(-1)^{n-j}x^{n-j}q^{\binom{n-j}{2}}/(q)_{n-j}\right]  .
\]
\qquad

In particular, taking $x\allowbreak=\allowbreak1-q$, we have%
\[
\left[  1/\left[  n-j\right]  _{q}!\right]  ^{-1}=\left[  (-1)^{n-j}%
q^{\binom{n-j+1}{2}}/\left[  n-j\right]  _{q}!\right]  .
\]
Further multiplying from the left-hand side by the matrix $\left[  \left\{
\left(  q\right)  _{n}\right\}  \right]  $ and from the right-hand side by the
matrix $\left[  \left\{  \left(  q\right)  _{j}^{-1}\right\}  \right]  $, we
get the following relationship:%
\begin{equation}
\left[  x^{n-j}%
\genfrac{[}{]}{0pt}{}{n}{j}%
_{q}\right]  ^{-1}=\left[  \left(  -x\right)  ^{n-j}q^{\binom{n-j}{2}}%
\genfrac{[}{]}{0pt}{}{n}{j}%
_{q}\right]  . \label{qbin}%
\end{equation}

2. Following (\ref{qnaw}), we get for all $x,y\in\mathbb{C}$:%
\[
\left[  y^{n-j}\left(  x\right)  _{n-j}/\left(  q\right)  _{n-j}\right]
^{-1}=\left[  y^{n-j}x^{n-j}\left(  1/x\right)  _{n-j}/\left(  q\right)
_{n-j}\right]
\]
and after applying similar trick with diagonal matrix multiplication we get%
\[
\left[  y^{n-j}%
\genfrac{[}{]}{0pt}{}{n}{j}%
_{q}\left(  x\right)  _{n-j}\right]  ^{-1}=\left[  y^{n-j}x^{n-j}%
\genfrac{[}{]}{0pt}{}{n}{j}%
_{q}\left(  1/x\right)  _{n-j}\right]
\]

\end{corollary}

\begin{proof}
We start with the assertion of Proposition \ref{qbin} and multiply both sides
of this identity from the left-hand side by $\left[  \left\{  \left(
q\right)  _{n}^{-1}\right\}  \right]  $ and from the right-hand side by
$\left[  \left\{  \left(  q\right)  _{j}\right\}  \right]  $ and then apply
assertion $4$. of Remark \ref{proste}. On the way we utilize the definition of
$%
\genfrac{[}{]}{0pt}{}{n}{j}%
_{q}$ given by (\ref{dqbin}).
\end{proof}

\begin{remark}
Notice that when passing with $q$ to $1$ in the last assertion of the
Corollary \ref{vqbin} we get (\ref{sila}).
\end{remark}

\begin{remark}
Notice also that some of the relationships mentioned in the Corollary
\ref{vqbin} refer to Riordan matrices of the Appell type. Namely, we have
$\left[  x^{n-j}/(q)_{n-j}\right]  \allowbreak=\allowbreak\left(  \left(
tx\right)  _{\infty}^{-1},t\right)  $ and obviously $\allowbreak\left[
(-1)^{n-j}x^{n-j}q^{\binom{n-j}{2}}/(q)_{n-j}\right]  \allowbreak
=\allowbreak\left(  \left(  tx\right)  _{\infty},t\right)  $. \newline$\left[
y^{n-j}\left(  x\right)  _{n-j}/\left(  q\right)  _{n-j}\right]
\allowbreak=\allowbreak\left(  \left(  txy\right)  _{\infty}/\left(
ty\right)  _{\infty},t\right)  $ as it follows from Proposition \ref{aux}, 1.
Somewhat less obvious is the relationship
\[
\left[  y^{n-j}x^{n-j}\left(  1/x\right)  _{n-j}/\left(  q\right)
_{n-j}\right]  \allowbreak=\allowbreak\left(  \left(  ty\right)  _{\infty
}/\left(  tyx\right)  _{\infty},t\right)  .
\]

\end{remark}

Recall, e.g., following \cite{Szab2020}, that the following polynomials
\[
R_{n}(x|q)\allowbreak=\allowbreak\sum_{j=0}^{n}%
\genfrac{[}{]}{0pt}{}{n}{j}%
_{q}x^{j},
\]
are called Rogers-Szeg\"{o} polynomials and they play an important r\^{o}le in
the $q-$series theory. Following, e.g., \cite{Szab2020}, we know that the
generating function of the polynomials $\left\{  R_{s}\right\}  $ are given by
the following formula:%
\[
\sum_{s\geq0}\frac{t^{s}}{\left(  q\right)  _{s}}R_{s}(x|q)=\frac{1}{\left(
t\right)  _{\infty}\left(  tx\right)  _{\infty}}.
\]
Hence, by simple operation of series multiplication and making use of the
formula (\ref{obinT}) we deduce that the following family of polynomials:%
\[
\hat{R}_{n}(x|q)=(-1)^{n}\sum_{s=0}^{n}%
\genfrac{[}{]}{0pt}{}{n}{s}%
_{q}x^{s}q^{\binom{s}{2}+\binom{n-s}{2}},
\]
defined for $n\geq0$ satisfy the following identity:%
\[
\sum_{j=0}^{n}%
\genfrac{[}{]}{0pt}{}{n}{j}%
_{q}R_{j}(x|q)\hat{R}_{n-j}(x|q)=\delta_{n,0}.
\]
Consequently, we have
\[
\left[  R_{n-j}(x|q)/\left(  q\right)  _{n-j}\right]  ^{-1}=\left[  \hat
{R}_{n-j}(x|q)/\left(  q\right)  _{n-j}\right]  .
\]

Following \cite{Szab2020}, formulae (3.16), (4.10), and (5.15), and the
proceeding each of these formulae definitions, we have%
\begin{align*}
\left[  h_{n-j}(x|q)/\left[  n-j\right]  _{q}!\right]  ^{-1}  &  =\left[
b_{n-j}(x|q)/\left[  n-j\right]  _{q}!\right]  ,~\\
\left[  h_{n-j}(x|a,q)/\left[  n-j\right]  _{q}!\right]  ^{-1}  &  =\left[
\hat{h}_{n-j}(x|a,q)/\left[  n-j\right]  _{q}!\right]  ,\\
\left[  Q_{n-j}(x|a,b,q)/\left[  n-j\right]  _{q}!\right]  ^{-1}  &  =\left[
\hat{Q}_{n-j}(x|a,b,q)/\left[  n-j\right]  _{q}!\right]  ,.
\end{align*}

where $\left\{  h_{n}(x|q\right\}  $, $\left\{  h_{n}(x|a,q)\right\}  $,
$\left\{  Q_{n}(x|a,b,q)\right\}  $, are respectively the so-called
$q-$Hermite, big $q-$Hermite, Al-Salam--Chihara polynomials. They constitute a
part of the so-called Askey-Wilson scheme of orthogonal polynomials defined by
their three-term recurrences given by the formulae respectively (3.1), (4.1),
(5.1) in \cite{Szab2020}. The families of polynomials $\left\{  b_{n}%
(x|q)\right\}  $, $\left\{  \hat{h}_{n}(x|a,q)\right\}  $, $\left\{  \hat
{Q}_{n}(x|a,b,q)\right\}  $ are defined by their three-term recurrences given
by the formulae respectively (3.14), (4.9) and unnumbered formula proceeding
(5.15). Although these families of polynomials were defined for real $x,a,b,q$
all from the segment $[-1,1]$, we can extend their ranges to all complex
numbers since they are polynomials.

\begin{remark}
Let us denote the following infinite product%
\[
\varphi_{h}(x|a)=1/\prod_{j=0}^{\infty}v(t|aq^{j}),
\]
defined for all $\left\vert t\right\vert ,\left\vert a\right\vert ,\left\vert
q\right\vert <1.$ where $v\left(  x|a\right)  \allowbreak=\allowbreak
1-2ax+a^{2}.$ It is known (see e.g. \cite{KLS}, or \cite{Szab22}, (3.6)) that
GF of the sequence $\left\{  h_{n}(x|q\right\}  $ is equal to $\varphi
_{h}(x|t).$ That is we have%
\[
\sum_{n\geq0}\frac{t^{n}}{\left(  q\right)  _{n}}h_{n}(x|q)=\varphi_{h}(x|t),
\]
defined for all $\left\vert t\right\vert ,\left\vert x\right\vert ,\left\vert
q\right\vert <1.$Similarly, following either \cite{KLS}, or \cite{Szab22}(4.7)
we see that the GF of the family $\left\{  h_{n}(x|a,q)\right\}  $ is $\left(
at\right)  _{\infty}\allowbreak\varphi_{h}(x|t).$ Finally following the same
references we observe that the GF of the family $\left\{  Q_{n}%
(x|a,b,q)\right\}  $ is $\left(  at,bt\right)  _{\infty}\allowbreak\varphi
_{h}(x|t).$ Hence, we observe that $\left[  h_{n-j}(x|q)/\left[  n-j\right]
_{q}!\right]  \allowbreak=\allowbreak\left[  h_{n-j}(x|q)/\left(  q\right)
_{n-j}\right]  \left[  \left\{  (1-q)^{j}\right\}  \right]  $ and consequently
that is a Riordan matrix having the form of a product $\left(  \varphi
_{h}(x|t),t\right)  \allowbreak\left(  1,(1-q)t\right)  .$ Similarly $\left[
h_{n-j}(x|a,q)/\left[  n-j\right]  _{q}!\right]  $ is a Riordan matrix of a
form $(\left(  at\right)  _{\infty}\allowbreak\varphi_{h}(x|t),t))\allowbreak
\left(  1,\left(  1-q\right)  t\right)  $ and $\left[  Q_{n-j}%
(x|a,b,q)/\left[  n-j\right]  _{q}!\right]  $ is also a Riordan matrix of a
form $\left(  \left(  at,bt\right)  _{\infty}\allowbreak\varphi_{h}%
(x|t),t\right)  \allowbreak\left(  1,\left(  1-q\right)  t\right)  .$ All
these considerations are true for $\left\vert t\right\vert ,\allowbreak
\left\vert x\right\vert ,\allowbreak\left\vert a\right\vert ,\allowbreak
\left\vert b\right\vert ,\allowbreak\left\vert q\right\vert <1.$
\end{remark}

\section{Glossary\label{GL}}

1. $\left[  a_{n,j}\right]  _{n,j\geq0,}$ with $a_{nj}\allowbreak
=\allowbreak0$ for all $j>n\geq0$, lower-triangular infinite matrix with
entries belonging to $\mathbb{C}$,

2. $\mathcal{S}$ -the algebra of lower- triangular matrices .

3.$\mathcal{D}$ -the sub-algebra of all diagonal matrices.

4. $\left[  \left\{  \beta_{n}\right\}  \right]  $ -diagonal matrix with
$\beta_{n}$ as its ($n+1)\times(n+1)$ entry.

5. $\mathcal{SE}$ -sub-ring of lower-triangular matrices $\left[
a_{n,j}\right]  $ with $a_{n,j}\allowbreak=\allowbreak0$ whenever $n-j$ is an
odd number.

6. $\mathcal{L}$ -the group of lower-triangular matrices.

7. $\mathcal{R}$ -the sub-group of Riordan matrices.

8. $\mathcal{IP}(p,\beta,\alpha)$, $\mathcal{P}(p,\alpha)$, $\mathcal{L(}%
p,a)$, $\mathcal{A}$, $\mathcal{C}$ -various Riordan subgroups defined in the
List of subgroups ...\ref{subgr} ,

9. $B_{n}(x)$, $E_{n}(x)$ -respectively Bernoulli and Euler polynomials
defined by (\ref{Bn}) and (\ref{En}).

10. $L_{n}(x)$, $He_{n}(x)$ -Laguerre and Hermite (more precisely the
so-called probabilistic Hermite polynomials)

11. $(x)_{(n)\text{ }}$, $\left(  x\right)  ^{\left(  n\right)  }$
respectively falling and rising factorials.

12. $\left(  x|q\right)  _{n}$ the so-called $q-$Pochhammer symbol (defined by
(\ref{qP}), often also denoted by $\left(  x\right)  _{n}$ when $q$ is well defined.

13. GF-generating functions.

14. RGF -reference generating function.

15 FPS -formal power series.

16. GBT($\alpha$) -generalized binomial transformation defined in Remark
\ref{simple}.

\part{Declarations}

\begin{itemize}
\item Availability of data and material: Not applicable,

\item Competing interests: Not applicable,

\item Funding: Not applicable,

\item Authors' contributions: I'm single author, so my contribution is 100\%,

\item Acknowledgements: Not applicable
\end{itemize}

\end{document}